\documentclass[12pt,psamsfonts]{amsart}
\usepackage[english]{babel}
\usepackage[utf8]{inputenc}
\usepackage{amsthm} 
\usepackage{amsmath}
\usepackage{amsfonts}
\usepackage{amssymb}
\usepackage{graphicx}
\usepackage{cite}
\usepackage{geometry}
\usepackage[colorlinks=true, linkcolor=cyan, citecolor=cyan, urlcolor=cyan]{hyperref}
\usepackage{hyperref}
\newtheorem{teorema}{Theorem}

\newtheorem{lema}{Lemma}

\newtheorem{Assumption}{Assumption}
\newtheorem{teoremaLetra}{Theorem}
 
\newcounter{remark}
\begin{document}
\title[BIFURCATION OF LIMIT CYCLES IN PIECEWISE HOLOMORPHIC SYSTEMS]{BIFURCATION OF LIMIT CYCLES IN PIECEWISE HOLOMORPHIC SYSTEMS SEPARATED BY TWO EXTERNAL CIRCLES}

\author{Carlos Vinícius das Neves Silva \and Paulo Ricardo da Silva}

\address{Departamento de Matem\'{a}tica -- Instituto de Bioci\^{e}ncias Letras
	e Ci\^{e}ncias Exatas, UNESP -- Univ Estadual Paulista, Rua C. Colombo, 2265,
	CEP 15054--000 S\~{a}o Jos\'{e} do Rio Preto, S\~{a}o Paulo, Brazil}
\email{carlos.vn.silva@unesp.br and paulo.r.silva@unesp.br}
\date{}


\date{}
\dedicatory{}
\maketitle

\begin{abstract}
This work investigates the number of limit cycles bifurcating from the periodic orbits of a linear center when it is perturbed by piecewise holomorphic systems with three zones separated by two external circles. We analyze how different geometric configurations of these zones influence the number of limit cycles. Moreover, we show the existence of at least sixteen limit cycles in fourth-degree piecewise polynomial holomorphic systems under certain configurations.
\end{abstract}
\section{Introduction}
One of the most studied problems in the qualitative theory of ordinary differential equations is Hilbert's $16$th problem \cite{refb15}. This problem seeks to determine the maximum number and arrangement of limit cycles that can exist in polynomial systems.

Originally, Hilbert's $16$th problem was proposed in the context of planar polynomial differential systems. More precisely, it asks for the maximum possible number of limit cycles that a polynomial differential system of degree $n$ can have, as well as the relative arrangement of these cycles in the plane. However, over the years, researchers have started to consider extensions of this problem to a broader class of systems, the so-called piecewise smooth systems \cite{refb12,refb13}. These systems consist of different vector fields defined in regions of the plane, with the boundaries between these regions given by curves.

Despite having been extensively investigated, both in the context of smooth systems and piecewise smooth systems, Hilbert's $16$th problem remains open. Over the years, different particular cases of this problem have been studied in the context of various classes of piecewise smooth systems, considering different geometries for the discontinuity set.

In the context of smooth systems, \cite{refb18} investigated the maximum number of limit cycles bifurcating from periodic orbits of a cubic isochronous center. In the context of piecewise smooth systems with two zones separated by a straight line, \cite{refb19} studied the limit cycles bifurcating from a linear center under perturbation by a piecewise linear system. For piecewise smooth systems with three zones separated by straight lines, \cite{refb7} obtained lower bounds for the number of limit cycles bifurcating from a periodic annulus in piecewise linear Hamiltonian systems, with a center in the central zone and saddles in the lateral zones.

When the discontinuity set is a circle, several recent works have investigated the number of limit cycles of piecewise smooth linear Hamiltonian systems. In \cite{refb20}, the authors studied piecewise linear systems formed by two linear centers separated by a circle, showing that such systems can have at most three limit cycles. More generally, when the discontinuity set is an arbitrary conic, \cite{refb21} determined the maximum number of crossing limit cycles, namely, cycles that cross the discontinuity curve at exactly two points, for systems formed by linear Hamiltonian saddles or linear centers.

For systems separated by two circles, \cite{refb11} studied linear centers separated by two concentric circles, considering limit cycles that cross the discontinuity set at four points, two on each circle. They showed that such systems admit at most three limit cycles and presented examples realizing one, two, and three cycles. In \cite{refb10}, the authors studied linear centers in the same configuration, but considered a distinct class of crossing cycles and obtained an upper bound of two limit cycles, with examples realizing this bound.

Another class of piecewise systems that has been investigated in the literature is that of piecewise holomorphic systems. Holomorphic systems $\dot{z}=f(z)$ have interesting properties: their phase portraits around zeros and poles are completely classified, and such systems are locally integrable \cite{refb2}. A notable fact about this class, however, is that they do not possess limit cycles. This restriction disappears when considering piecewise holomorphic systems: \cite{refb5} exhibit piecewise holomorphic systems separated by a straight line that possess limit cycles. For this same class of systems, \cite{refb6} obtained lower bounds for the number of limit cycles when the discontinuity set is a straight line, combining different approaches: the study of the zeros of averaging functions and a degenerate Andronov--Hopf bifurcation; and \cite{refb1} studied the bifurcation of limit cycles in two distinct periodic annuli, around two centers of the same perturbed piecewise holomorphic system.

When the discontinuity set is a circle, \cite{refb22} obtained, by means of second-order averaging theory, lower bounds for the number of limit cycles of piecewise holomorphic polynomial systems: for perturbations of degrees $n_{+},n_{-}$, they showed that the number of limit cycles is at least $\lfloor(3\max\{n_+,n_-\}+5)/2\rfloor$. For the case of two circles, \cite{refb9} studied the limit cycles bifurcating from a linear center perturbed by piecewise holomorphic systems with three zones separated by two tangent circles in two geometric configurations: external tangency, for which, perturbing by degree-$4$ polynomials, they obtained a lower bound of $5$ limit cycles, and internal tangency, in which, perturbing by degree-$3$ polynomials and depending on the family of periodic orbits considered (crossing only one of the circles, or both), they obtained lower bounds of $4$ and $8$ limit cycles, respectively.

In this work, we consider the study of piecewise holomorphic systems with three zones, of the form
\begin{eqnarray*}
    \dot{z}=f_{i}(z), \quad i=1,2,3,
\end{eqnarray*}
where $f_{i}$ are polynomials in the variable $z$. The three zones of the plane are delimited by \textbf{two external circles}, that is, two circles that do not intersect and such that neither is contained in the interior of the other. The main objective is to analyze the number of limit cycles that bifurcate from the linear center $\dot{z}=iz$ when it is subjected to perturbations by piecewise holomorphic systems separated by two external circles. In this way, this work extends the analysis of piecewise holomorphic systems separated by two circles to a geometric configuration that has not yet been considered in the literature, contributing to a more comprehensive understanding of how the geometry of the discontinuity set affects the bifurcation of limit cycles from a linear center.

\subsection{Our main results}
In this section, we state the main results of this paper. Table \ref{tab1} summarizes the three geometric configurations considered, together with the corresponding number of families of periodic orbits, the lower bounds obtained for each family, and the total number of limit cycles obtained in each configuration through degree-$4$ polynomial perturbations, considering $h_{1},h_{2},h_{3}$ as in (\ref{eq.it1})-(\ref{eq.it3}), with $n=4$.

\begin{table}[!htb]
\centering
\resizebox{\textwidth}{!}{%
\begin{tabular}{c l l c c}
\hline
Theorem & Configuration & Families & Bounds & Example (degree $4$) \\
\hline
A & equal radii, aligned & $1$ (crosses both) & $6$ & $6$ \\
B & different radii, aligned & $2$ (same circle) $+\,1$ (crosses both) & $5,5,10$ & $16$ $(3+3+10)$ \\
C & non-aligned & $2$ (one circle each) $+\,1$ (crosses both) & $5,5,10$ & $14$ $(2+10+2)$ \\
\hline
\end{tabular}%
}
\caption{Summary of the geometric configurations and lower bounds obtained in Theorems \ref{teoa}, \ref{teob} and \ref{teoc}. }
\label{tab1}
\end{table}

We begin with the analysis of the case of highest symmetry, in which the two discontinuity circles have identical radii and their centers are aligned with the linear center of the unperturbed system. Under such conditions, the unperturbed system presents a single annulus of periodic orbits, where the trajectories within this annulus intersect both regions of discontinuity.
\begin{teoremaLetra}\label{teoa}
  Consider the piecewise polynomial holomorphic system of the form
  \begin{eqnarray}\label{eq.it1}
\begin{cases}
\dot{z}=iz+\varepsilon h_{1}(z) \quad \mbox{if} \quad |z+2|<1,\\
\dot{z}=iz+\varepsilon h_{2}(z)\quad \mbox{if} \quad |z+2|>1 \quad \hbox{and }\quad |z-2|>1  \\
\dot{z}=iz+\varepsilon h_{3}(z) \quad \mbox{if} \quad |z-2|<1,
\end{cases}
\end{eqnarray}
where
\begin{eqnarray*}
    h_{1}(z)=\sum_{j=0}^{4}(a_{1_{j}}+ib_{1_{j}})z^{j}, h_{2}(z)=\sum_{j=0}^{4}(a_{2_{j}}+ib_{2_{j}})z^{j} \; \hbox{and} \; h_{3}(z)=\sum_{j=0}^{4}(a_{3_{j}}+ib_{3_{j}})z^{j},
\end{eqnarray*}
we obtain that:
\begin{enumerate}
    \item[a)]  The lower bound for the number of limit cycles bifurcating from the periodic solutions, for $\varepsilon$ sufficiently small is $6$.
    \item[b)] There exist parameters such that the system (\ref{eq.it1}) has $6$ limit cycles. 
\end{enumerate}
\end{teoremaLetra}

 While in Theorem \ref{teoa} we consider circles with equal radii, in the next theorem we analyze the system when the radius of one of the circles is altered. This modification affects the structure of the periodic orbits: instead of a single annulus, the system now exhibits three annuli of periodic orbits. In this new configuration, one annulus consists of orbits that intersect both circles, while the other two annuli consist of orbits that intersect only one of the circles (the same circle for both annuli). This structural change in the discontinuity regions has implications for the number of limit cycles.

\begin{teoremaLetra}\label{teob}
Consider the piecewise polynomial holomorphic system of the form
      \begin{eqnarray}\label{eq.it2}
\begin{cases}
\dot{z}=iz+\varepsilon h_{1}(z) \quad \mbox{if} \quad |z+2|<1,\\
\dot{z}=iz+\varepsilon h_{2}(z)\quad \mbox{if} \quad |z+2|>1 \quad \hbox{and }\quad |z-2|>\sqrt{2}  \\
\dot{z}=iz+\varepsilon h_{3}(z) \quad \mbox{if} \quad |z-2|<\sqrt{2},
\end{cases}
\end{eqnarray}
where
\begin{eqnarray*}
    h_{1}(z)=\sum_{j=0}^{4}(a_{1_{j}}+ib_{1_{j}})z^{j}, h_{2}(z)=\sum_{j=0}^{4}(a_{2_{j}}+ib_{2_{j}})z^{j} \; \hbox{and} \; h_{3}(z)=\sum_{j=0}^{4}(a_{3_{j}}+ib_{3_{j}})z^{j},
\end{eqnarray*}
we obtain that:
\begin{enumerate}
    \item[a)]  The lower bound for the number of limit cycles bifurcating from the periodic solutions, for $\varepsilon$ sufficiently small, is $5$ for each of the two families of periodic orbits that cross only $|z-2|=\sqrt{2}$, and $10$ for the periodic orbits that cross both $|z+2|=1$ and $|z-2|=\sqrt{2}$.
    \item[b)] There exist parameters such that the system (\ref{eq.it2}) has $16$ limit cycles. 
\end{enumerate}
\end{teoremaLetra}

Statement (a) considers each annulus separately. Therefore, the parameters that maximize the number of limit cycles in one annulus need not be the same as those in the others. On the other hand, statement (b) uses a single choice of parameters that allows limit cycles to bifurcate simultaneously from all three annuli, resulting in the example of sixteen limit cycles.

After analyzing Theorem \ref{teob}, we study in Theorem \ref{teoc} the case where the center of one of the circles is no longer aligned with the line formed by the linear center and the center of the other circle. In this scenario, we also find three annuli of periodic orbits, one where the orbits intersect both circles, and two others where each annulus intersects a different circle. This loss of alignment introduces new asymmetries that further influence the number of limit cycles.

\begin{teoremaLetra}\label{teoc}
    Consider the piecewise polynomial holomorphic system of the form
      \begin{eqnarray}\label{eq.it3}
\begin{cases}
\dot{z}=iz+\varepsilon h_{1}(z) \quad \mbox{if} \quad |z+2|<1,\\
\dot{z}=iz+\varepsilon h_{2}(z)\quad \mbox{if} \quad |z+2|>1 \quad \hbox{and }\quad |z-2-i|>1  \\
\dot{z}=iz+\varepsilon h_{3}(z) \quad \mbox{if} \quad |z-2-i|<1,
\end{cases}
\end{eqnarray}
where
\begin{eqnarray*}
    h_{1}(z)=\sum_{j=0}^{4}(a_{1_{j}}+ib_{1_{j}})z^{j}, h_{2}(z)=\sum_{j=0}^{4}(a_{2_{j}}+ib_{2_{j}})z^{j} \; \hbox{and} \; h_{3}(z)=\sum_{j=0}^{4}(a_{3_{j}}+ib_{3_{j}})z^{j},
\end{eqnarray*}
we obtain that:
\begin{enumerate}
    \item[a)]  The lower bound for the number of limit cycles bifurcating from the periodic solutions, for $\varepsilon$ sufficiently small is $5$ for the periodic orbits that cross only $|z+2|=1$, $10$ for the periodic orbits that cross both $|z+2|=1$ and $|z-2-i|=1$, and $5$ for the periodic orbits that cross only $|z-2-i|=1$.
    \item[b)] There exist parameters such that the system (\ref{eq.it3}) has $14$ limit cycles. 
\end{enumerate}
\end{teoremaLetra}

As in Theorem \ref{teob}, the bounds $5,10$ and $5$ refer separately to three annuli and cannot be added. Statement (b), by contrast, uses a single common perturbation resulting in fourteen limit cycles.
\begin{figure}[!htb]
\centering
\includegraphics[scale=0.40]{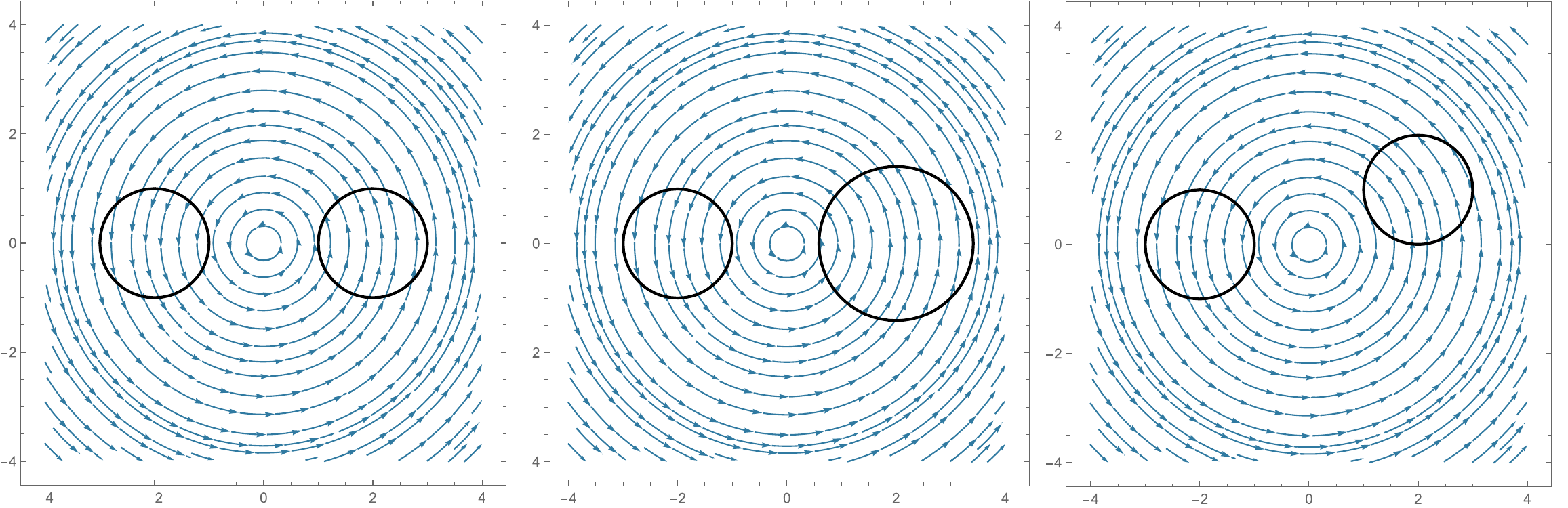}
\caption{Phase portraits of systems (\ref{eq.it1}), (\ref{eq.it2}) and (\ref{eq.it3}), for $\varepsilon=0$.}
\label{img1(s2).in}
\end{figure}

The proofs of these theorems are based on an analysis of the Melnikov function for the corresponding piecewise systems. Through estimating its zeros, we are able to estimate the number of limit cycles emerging from the perturbation of the center.

The structure of the paper is as follows: In Section \ref{sec2}, the preliminaries are presented; in Section \ref{sec3}, we provide the proof of Theorem \ref{teoa}; in Section \ref{sec4}, we present the proof of Theorem \ref{teob}; in Section \ref{sec5}, we prove Theorem \ref{teoc}; finally, Appendix \ref{sec6} contains the proof of Theorem \ref{teomelnikov}, and Appendix \ref{app.w} collects the Wronskian computations.

\section{Preliminaries}\label{sec2}

This section provides initial results that play a key role in the development of the rest of the study.
\subsection{Melnikov function}
 We will review some results on Melnikov functions for piecewise smooth systems, based on \cite{refb8}, \cite{refb14}, \cite{refb16} and \cite{refb17}, adapting them to our setting with different configurations of discontinuities. By following these works, we apply the same methodology while preserving the structure of the arguments. The detailed arguments are deferred to Appendix \ref{sec6}.

Let $\Sigma_{1}, \Sigma_{2},\tilde{\Sigma }_{1}$ and $\tilde{\Sigma }_{2}$ be smooth curves such that $\Sigma =\Sigma_{1}\cup \Sigma_{2}$ and $\tilde{\Sigma}=\tilde{\Sigma}_{1}\cup \tilde{\Sigma}_{2}$, and  $\Sigma $ and $\tilde{\Sigma }$ divide the plane $\mathbb{R}^{2}$ into three open and connected sets denoted by $S_{1},S_{2}$ and $S_{3}$.

Consider a piecewise smooth system
\begin{eqnarray}\label{melnikov1}
\begin{cases}
\dot{x}=H_{y}(x,y)+\varepsilon f(x,y),\\
\dot{y}=-H_{x}(x,y)+\varepsilon g(x,y),
\end{cases}
\end{eqnarray}
where $\varepsilon $ is a small parameter,
\[
\begin{array}{rl}
f(x,y) = 
  \begin{cases}
    f_1(x,y), & (x,y) \in S_{1} \\[6pt]
    f_2(x,y), & (x,y) \in S_{2} \\[6pt]
    f_3(x,y), & (x,y) \in S_{3}
  \end{cases}
\quad , \quad
g(x,y) = 
  \begin{cases}
    g_1(x,y), & (x,y) \in S_{1} \\[6pt]
    g_2(x,y), & (x,y) \in S_{2} \\[6pt]
    g_3(x,y), & (x,y) \in S_{3}
  \end{cases}
\end{array}
\]
and
\[
H(x,y) = 
  \begin{cases}
    H_1(x,y), & (x,y) \in S_{1} \\[6pt]
    H_2(x,y), & (x,y) \in S_{2} \\[6pt]
    H_3(x,y), & (x,y) \in S_{3}
  \end{cases}
\]
where $f_{i}(x,y),g_{i}(x,y)$ and $H_{i}(x,y)$ for $i=1,2,3$ are analytic functions.

Suppose that the unperturbed system (\ref{melnikov1}) (when $\varepsilon =0 $) satisfies the following assumptions:

\begin{Assumption}\label{Assumption1}
There exists a family of periodic orbits $L_{h}$, where $h\in J=(a,b)$  is a non-empty interval. In particular, each $L_{h}$ in $S_{1}$ is defined by $H_{1}(x,y)=h$.
\end{Assumption}
\begin{Assumption}\label{Assumption2} Each $L_{h}$ intersects transversely the curves  $\Sigma_{i}$ and $\tilde{\Sigma}_{i}$, intersecting each of them in exactly one point: $A_{i}(h)\in \Sigma _{i}$ and $\tilde{A}_{i}(h)\in \tilde{\Sigma }_{i}$ for $i=1,2$. Furthermore, $L_{h}$ crosses $\Sigma_{i}$ and $\tilde{\Sigma}_{i}$ in the counterclockwise direction.
\end{Assumption}

Under the  Assumptions \ref{Assumption1} and \ref{Assumption2}, the unperturbed system has a family of periodic orbits, $L_{h}=L_{h}^{1}\cup L _{h}^{2}\cup L_{h}^{3}\cup L_{h}^{4}$, for $h\in J$, where $L_{h}^{1}$ is an orbit that starts at $A_{1}(h)$ and ends at $A_{2}(h)$, $L_{h}^{2}$ is an orbit that starts at $A_{2}(h)$ and ends at $\tilde{A}_{2}(h)$, $L_{h}^{3}$ is an orbit that starts at  $\tilde{A}_{2}(h)$ and ends at  $\tilde{A}_{1}(h)$, and $L_{h}^{4}$ is an orbit that starts at  $\tilde{A}_{1}(h)$ and ends at  $A_{1}(h) $ (as illustrated in Figure \ref{img1(s2)}).

If $\Sigma_{1}=\{ (x,y);\varphi_{1}(x,y)=0\}$, $\Sigma_{2}=\{ (x,y);\varphi_{2}(x,y)=0\}$, $\tilde{\Sigma}_{1}=\{ (x,y);\tilde{\varphi}_{1}(x,y)=0\}$ and $\tilde{\Sigma}_{2}=\{ (x,y);\tilde{\varphi}_{2}(x,y)=0\}$, with $\varphi_{1},\varphi_{2},\tilde{\varphi }_{1},\tilde{\varphi }_{2}\in C^{1}(\mathbb{R}^{2},\mathbb{R})$, then:
\begin{teorema}\label{teomelnikov}
    Consider system (\ref{melnikov1}) with $\varepsilon $ sufficiently small. Under the Assumptions \ref{Assumption1} and \ref{Assumption2}, the first-order Melnikov function can be expressed as
\begin{equation}
\begin{aligned}
M(h) &= \Psi_{1}(h) \int_{L_{h}^{4}} (g_{2}\,dx - f_{2}\,dy)
     + \Psi_{2}(h) \int_{L_{h}^{3}} (g_{3}\,dx - f_{3}\,dy) \\
     &\quad + \Psi_{3}(h) \int_{L_{h}^{2}} (g_{2}\,dx - f_{2}\,dy)
     + \Psi_{4}(h) \int_{L_{h}^{1}} (g_{1}\,dx - f_{1}\,dy)
\end{aligned}
\label{funcaomelnikov}
\end{equation}
where
 \begin{equation*}
    \begin{aligned}
        \Psi_{1}(h)&=  \frac{\langle \nabla H_{1}(A_{1}), \nabla \varphi_{1}^{\perp}(A_{1}) \rangle}{\langle \nabla H_{2}(A_{1}), \nabla \varphi_{1}^{\perp}(A_{1}) \rangle} ,\\
        \Psi_{2}(h) &= \frac{\langle \nabla H_{1}(A_{1}), \nabla \varphi_{1}^{\perp}(A_{1}) \rangle \langle \nabla H_{2}(\tilde{A}_{1}), \nabla \tilde{\varphi}_{1}^{\perp}(\tilde{A}_{1}) \rangle }{\langle \nabla H_{2}(A_{1}), \nabla \varphi_{1}^{\perp}(A_{1}) \rangle \langle \nabla H_{3}(\tilde{A}_{1}), \nabla \tilde{\varphi}_{1}^{\perp}(\tilde{A}_{1}) \rangle } \\
        \Psi_{3}(h)& =\frac{\langle \nabla H_{1}(A_{1}), \nabla \varphi_{1}^{\perp}(A_{1}) \rangle \langle \nabla H_{2}(\tilde{A}_{1}), \nabla \tilde{\varphi}_{1}^{\perp}(\tilde{A}_{1}) \rangle \langle \nabla H_{3}(\tilde{A}_{2}), \nabla \tilde{\varphi}_{2}^{\perp}(\tilde{A}_{2}) \rangle  }{\langle \nabla H_{2}(A_{1}), \nabla \varphi_{1}^{\perp}(A_{1}) \rangle \langle \nabla H_{3}(\tilde{A}_{1}), \nabla \tilde{\varphi}_{1}^{\perp}(\tilde{A}_{1}) \rangle \langle \nabla H_{2}(\tilde{A}_{2}), \nabla \tilde{\varphi}_{2}^{\perp}(\tilde{A}_{2}) \rangle}  \\
        \Psi_{4}(h)& =\frac{\langle \nabla H_{1}(A_{1}), \nabla \varphi_{1}^{\perp}(A_{1}) \rangle \langle \nabla H_{2}(\tilde{A}_{1}), \nabla \tilde{\varphi}_{1}^{\perp}(\tilde{A}_{1}) \rangle \langle \nabla H_{3}(\tilde{A}_{2}), \nabla \tilde{\varphi}_{2}^{\perp}(\tilde{A}_{2}) \rangle  \langle \nabla H_{2}(A_{2}), \nabla \varphi_{2}^{\perp}(A_{2}) \rangle}{\langle \nabla H_{2}(A_{1}), \nabla \varphi_{1}^{\perp}(A_{1}) \rangle \langle \nabla H_{3}(\tilde{A}_{1}), \nabla \tilde{\varphi}_{1}^{\perp}(\tilde{A}_{1}) \rangle \langle \nabla H_{2}(\tilde{A}_{2}), \nabla \tilde{\varphi}_{2}^{\perp}(\tilde{A}_{2}) \rangle \langle \nabla H_{1}(A_{2}), \nabla \varphi_{2}^{\perp}(A_{2}) \rangle} 
    \end{aligned}   
\end{equation*}

    Furthermore, if $M(h)=0$ and $M'(h)\neq 0$ for some $h\in J$, then for $\varepsilon $ sufficiently small, the system  (\ref{melnikov1}) has a unique limit cycle bifurcating from $L_{h}$.
\end{teorema}

\begin{figure}[!htb]
\centering
\includegraphics[scale=1.0]{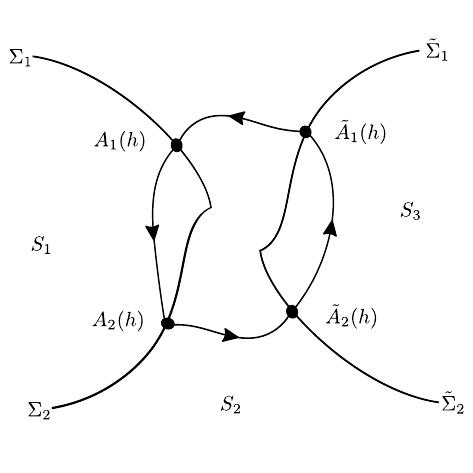}
\caption{The closed orbits of the unperturbed system (\ref{melnikov1}).}
\label{img1(s2)}
\end{figure}

\refstepcounter{remark}
\label{obs.1} 
\noindent\textbf{Remark  \theremark.} 
The Assumptions \ref{Assumption1} and \ref{Assumption2} guarantee the existence of a family of periodic orbits that intersect the curves $\Sigma $ and $\tilde{\Sigma }$. However, it is also possible to have families of periodic orbits that intersect only one of these curves. To consider these cases, without  loss of generality, we describe the case of orbits that intersect only $\Sigma $ and state that the case of orbits that intersect only $\tilde{\Sigma }$ follows in an analogous manner (this case is treated in \cite{refb19}).
\begin{enumerate}
    \item There exists a family of periodic orbits $l_{h}$, where $h\in I_{1}=(a_{1},b_{1})$ is a non-empty interval. In particular, each $l_{h}$ in $S_{1}$ is defined by $H_{1}(x,y)=h$.
    \item  Each $l_{h}$ intersects transversely the curves $\Sigma_{i}$ intersecting each of them in exactly one point $a_{i}(h)\in\Sigma_{i}$ for $i=1,2$. Furthermore, $l_{h}$ crosses $\Sigma_{i}$ in the counterclockwise direction.
\end{enumerate}
    In this case, the Melnikov function takes the form:
    \begin{eqnarray*}
        M_{1}(h)&=&\dfrac{\langle \nabla H_{1}(a_{1}),\nabla \varphi_{1}^{\perp}(a_{1}) \rangle}{\langle \nabla H_{2}(a_{1}),\nabla \varphi_{1}^{\perp}(a_{1}) \rangle}\int_{l_{h}^{1}}(g_{2}dx-f_{2}dy)+\\
        & &\dfrac{\langle \nabla H_{1}(a_{1}),\nabla \varphi_{1}^{\perp}(a_{1}) \rangle \langle \nabla H_{2}(a_{2}),\nabla \varphi_{2}^{\perp}(a_{2}) \rangle}{\langle \nabla H_{2}(a_{1}),\nabla \varphi_{1}^{\perp}(a_{1}) \rangle \langle \nabla H_{1}(a_{2}),\nabla \varphi_{2}^{\perp}(a_{2}) \rangle}\int_{l_{h}^{2}}(g_{1}dx-f_{1}dy)
    \end{eqnarray*}

Furthermore, if $M_{1}(h)=0$ and $M_{1}'(h)\neq 0$ for some $h\in I_{1}$, then for sufficiently small $\varepsilon $, the system (\ref{melnikov1}) has a unique limit cycle bifurcating from $l_{h}$, where $l_{h}=l_{h}^{1}\cup l_{h}^{2}$.

\refstepcounter{remark}
\label{obs.2} 
\noindent\textbf{Remark  \theremark.} 
In particular, consider the system 
\begin{eqnarray*}
\begin{cases}
\dot{z}=A(z)+\varepsilon B_{1}(z) \quad \mbox{if}\quad z\in S_{1},\\
\dot{z}=A(z)+\varepsilon B_{2}(z)\quad \mbox{if} \quad  z\in S_{2} \\
\dot{z}=A(z)+\varepsilon B_{3}(z) \quad \mbox{if} \quad z\in S_{3},
\end{cases}
\end{eqnarray*}
where $A(x+iy)=H_{y}-iH_{x}$ and the unperturbed system satisfies Assumptions \ref{Assumption1} and \ref{Assumption2}. In this case, the Melnikov function can be expressed as
\[
    M(h)=\Re\left(\int_{L_{h}^{1}}i\overline{B}_{1}(z)dz \right)+\Re\left(\int_{L_{h}^{2}}i\overline{B_{2}}(z) dz\right)\]\[+
    \Re\left(\int_{L_{h}^{3}}i\overline{B_{3}}(z) dz\right)+\Re\left(\int_{L_{h}^{4}}i\overline{B_{2}}(z) dz\right).
\]  

Assuming that the unperturbed system satisfies the hypotheses of Remark \ref{obs.1}, the Melnikov function can be expressed as
\[
    M_{1}(h)=\Re\left(\int_{l_{h}^{1}}i\overline{B}_{2}(z)dz \right)+\Re\left(\int_{l_{h}^{2}}i\overline{B_{1}}(z) dz\right).
\] 

In fact, it is sufficient to observe that if $B=u+iv$ is a continuous function and $\gamma $ is a smooth curve, then
\begin{eqnarray*}
    \int_{\gamma }i\overline{B}dz=\int_{\gamma }(vdx-udy)+i\int_{\gamma }(vdy+udx).
\end{eqnarray*}

\subsection{Zero-counting results}
    We begin this subsection by presenting a result from \cite{ref3}, which will be used together with two additional results from \cite{ref4}. These three results will be fundamental for estimating the number of simple zeros of the Melnikov function.
\begin{lema}\label{lem1}
Consider $n+1$ linearly independent functions $f_{i}:U\subset \mathbb{R}\rightarrow \mathbb{R}$, $i=0,\cdots n$, then:
\begin{enumerate}
\item[a)] Given $n+1$ arbitrary values $x_{i}\in U$, $i=0,\cdots ,n$, there exist $n+1$ constants $C_{i}$, $i=0,\cdots n$ such that
\begin{eqnarray}\label{ar.eq1}
f(x):=\sum_{i=0}^{n}C_{i}f_{i}
\end{eqnarray}
is not the zero function and $f(x_{i})=0$, $i=0,\cdots ,n$;
\item[b) ] Furthermore, if all $f_{i}$ are analytic functions on $U$ and there exists $0\leq j\leq n$ such that $f_{j} \big|_U$ has a constant sign, it is possible to get an $f$ by (\ref{ar.eq1}), such that it has at least $n$ simple zeros in $U$.
\end{enumerate}
\end{lema}

Let $f_{0},\cdots ,f_{n}$ be functions that are differentiable on an interval $I$. Determining whether $f_{0},\cdots ,f_{n}$ are linearly independent is not a trivial task. A useful tool to determine if $f_{0},\cdots ,f_{n}$ are linearly independent is the Wronskian, which is defined by
\begin{eqnarray*}
W(f_0, f_1, \ldots, f_{n})(x)=\hbox{det}(M(f_{0},\cdots ,f_{n})(x)),
\end{eqnarray*}
where
\begin{eqnarray*}
M(f_{0},\cdots ,f_{n})(x)= \begin{bmatrix}
f_0(x) & f_1(x) & \cdots & f_{n}(x) \\
f_0'(x) & f_1'(x) & \cdots & f_{n}'(x) \\
\vdots & \vdots & \ddots & \vdots \\
f_0^{(n)}(x) & f_1^{(n)}(x) & \cdots & f_{n}^{(n)}(x)
\end{bmatrix}.
\end{eqnarray*}

If $W(f_0, f_1, \ldots, f_{n})(x)\neq 0$ for some $x\in I$, then $f_{0},\cdots ,f_{n}$ are linearly independent.
Let $\mathcal{F}=[f_{0},f_{1},\ldots,f_{n}]$ be an ordered set of analytic functions defined on a closed interval $[a,b]$. We denote by $Z(\mathcal{F})$ the maximum number of zeros, counting multiplicities, that any nontrivial function $v\in \operatorname{Span}(\mathcal{F})$ can have on $[a,b]$.
If $Z(\mathcal{F})\leq n$, then $\mathcal{F}$ is called an extended Chebyshev system, or an ET-system, on $[a,b]$.
We say that $\mathcal{F}$ is an extended complete Chebyshev system, or an ECT-system, on $[a,b]$ if, for every $k=0,\ldots,n$, the ordered subset $[f_{0},f_{1},\ldots,f_{k}]$ is an ET-system on $[a,b]$. Equivalently,
\begin{eqnarray*}
    Z([f_{0},f_{1},\ldots,f_{k}])\leq k, \qquad k=0,\ldots,n.
\end{eqnarray*}
\begin{lema}\label{lema2}
Let $\mathcal{F}=[f_{0},f_{1},\ldots,f_{n}]$ be an ordered set of analytic functions on $[a,b]$. Then:
    \begin{itemize}
        \item[a)] $\mathcal{F}$ is an ECT-system on $[a,b]$ if and only if
        \begin{eqnarray*}
            W_{k}(t)=W(f_{0},\ldots,f_{k})(t)\neq 0
        \end{eqnarray*}
        for every $t\in[a,b]$ and every $k=0,\ldots,n$.
        \item[b)] If $\mathcal{F}$ is an ECT-system on $[a,b]$, then every nontrivial element of $\operatorname{Span}(\mathcal{F})$ has at most $n$ zeros on $[a,b]$, counting multiplicities. Moreover, for every configuration of $m\leq n$ zeros, counting multiplicities, there exists an element of $\operatorname{Span}(\mathcal{F})$ realizing that configuration.
    \end{itemize}
\end{lema}
\begin{lema}\label{lema3}
Let $\mathcal{F}=[f_{0},f_{1},\ldots,f_{n}]$ be an ordered set of $C^{\infty}$ functions on $[a,b]$. Assume that $W_{k}(t)\neq0$, $k=0,\ldots,n-1$, for every $t\in[a,b]$, and that $W_{n}$ has exactly one zero on $[a,b]$, which is simple. Then $Z(\mathcal{F})=n+1$. Moreover, for every configuration of $m\leq n+1$ zeros, counting multiplicities, there exists an element of $\operatorname{Span}(\mathcal{F})$ realizing that configuration.
\end{lema}
\refstepcounter{remark}
\label{obs.3}
\noindent\textbf{Remark  \theremark.}
Lemmas \ref{lema2} and \ref{lema3} require $\mathcal{F}$ to be analytic (resp. $C^{\infty}$) on the \emph{closed} interval $[a,b]$. In Sections \ref{sec3}--\ref{sec5}, the sets $\mathcal{F}$ are built from angle functions $\theta_{r}^{i}=\arccos(\cdots)$ that fail to be differentiable exactly at the endpoints of their natural $r$-intervals --- precisely the points where the level curve $|z|=r$ becomes tangent to a switching circle instead of crossing it transversely. Consequently, each application of Lemma \ref{lema2} or Lemma \ref{lema3} below is made on a closed subinterval $[a,b]$ chosen strictly inside the relevant open interval.

\subsection{An integral formula for circular arcs}
While the results of the previous subsection are used to count the zeros of a given linear combination of functions, the result of this subsection provides the fundamental method for computing that linear combination: a closed-form expression for an anti-derivative of the real part of a polynomial $1$-form along an arc of a circle centered at the origin. This result is used in all the Melnikov-function computations in Sections \ref{sec3}-\ref{sec5}.
\begin{lema}\label{lemaintegral}
Let $P(z)=\sum_{j=0}^{n}(a_{j}+ib_{j})z^{j}$ be a polynomial function. Consider the curve $\gamma :[a,b]\to \mathbb{C}$ given by $\gamma (t)=re^{it}$, where $r=\sqrt{2h}$ and $h>0$. Then, the following identity holds:
\begin{eqnarray*}
    \Re \left( \int_{\gamma }i\overline{P(z)}dz\right)=\bigg[ -a_{0} r \sin t + b_{0} r \cos t - a_{1} r^{2} t + \sum_{j=2}^{n} \frac{r^{j+1}}{j-1} C_{j}(t) \bigg]_{a}^{b},
\end{eqnarray*}
where $C_j(t) = -a_j \sin((j-1)t) - b_j \cos((j-1)t)$.
\end{lema}
\begin{proof}
    First, note that 
\begin{eqnarray*}
i\overline{P(z)}=i\sum_{j=0}^{n} (a_{j}-ib_{j})\overline{z}^{j}=\sum_{j=0}^{n}(b_{j}+ia_{j})\overline{z}^{j}.
\end{eqnarray*}

We have
\begin{eqnarray*}
\int_{\gamma } i \overline{P} \, dz &=& \int_{\gamma } \sum_{j=0}^{n} (b_j + i a_j) \overline{z}^j \, dz \\
&=& \int_{a}^{b} \left( \sum_{j=0}^{n} (b_j + i a_j) r^j \overline{e^{ijt}} \right) i r e^{it} \, dt \\
&=& \int_{a}^{b} \left( \sum_{j=0}^{n} (-a_j + i b_j) r^{j+1} e^{-i(j-1)t} \right) dt.
\end{eqnarray*}

Before proceeding with the calculation of the integral above, we define some functions that will appear throughout the computation, in order to simplify the notation. These functions are given by
\begin{eqnarray*}
    A_j(t) &=& -a_j \cos((j-1)t) + b_j \sin((j-1)t), \\
    B_j(t) &=& a_j \sin((j-1)t) + b_j \cos((j-1)t). 
\end{eqnarray*}

Using this notation, it follows that
\begin{eqnarray*}
    (-a_{j} + i b_{j}) e^{-i(j-1)t} = A_{j}(t) + i B_{j}(t).
\end{eqnarray*}

As a consequence
\begin{align*}
\int_{\gamma } i\, \overline{P}\, dz 
&= \int_{a}^{b} 
\left( -a_{0} r \cos t - b_{0} r \sin t - a_{1} r^{2} 
+ \sum_{j=2}^{n} \left( r^{j+1}A_{j}(t )\right) \right) dt \\
&\quad + i  \int_{a}^{b}
\left( b_{0} r \cos t - a_{0} r \sin t + b_{1} r^{2} 
+ \sum_{j=2}^{n} \left( r^{j+1} B_{j}(t ) \right) \right) dt .
\end{align*}

Thus
\begin{eqnarray*}
\Re \left( \int_{\gamma } i \overline{P} \, dz \right) &=& \int_{a}^{b} \left( -a_{0} r \cos t - b_{0} r \sin t - a_{1} r^{2} + \sum_{j=2}^{n} r^{j+1} A_{j}(t) \right) dt \\
&=& \bigg[ -a_{0} r \sin t + b_{0} r \cos t - a_{1} r^{2} t + \sum_{j=2}^{n} \frac{r^{j+1}}{j-1} C_{j}(t) \bigg]_{a}^{b}.
\end{eqnarray*}
\end{proof}
\noindent\emph{General computational scheme.} Throughout Sections \ref{sec3}-\ref{sec5}, for a family of periodic orbits $L_{h}$ that crosses one or two circles, the Melnikov function is obtained and its zeros are then estimated by the same three-step scheme, which we record once here instead of repeating it in each proof:
\begin{enumerate}
    \item[(i)] \emph{Decomposition.} The curve $L_h$ is decomposed into its constituent circular arcs, each bounded by intersection angles determined by the geometry of the circles. The corresponding Melnikov function is then obtained as a sum of the integrals along each arc
    \begin{eqnarray*}
       \sum_{k}\Re\left( \int_{\mathrm{arc}_k}i\overline{h{_k}}dz \right)
    \end{eqnarray*}
    \item[(ii)] \emph{Evaluation.} Each term is computed in closed form via Lemma \ref{lemaintegral}, and the resulting expression is collected as a linear combination $\sum_{j}c_{j}f_{j}$ of a set $\mathcal{F}=\{f_{j}(r)\}$ built from the elementary functions $r^{k}\sin (m\theta_{r}^{k}), r^{k}\cos (m\theta_{r}^{k}), r^{2}$ and $r^{2}\theta_{r}^{k}$, where the coefficients $c_{j}$ are explicit linear functions of the free coefficients of $h_{1}, h_{2},h_{3}$.
    \item[(iii)] \emph{Zero Count.} Two approaches are used, depending on how much can be established about $\mathcal{F}$: if the Wronskians of $\mathcal{F}$ (Appendix \ref{app.w}) show that it is an ECT-system on the relevant $r$-interval, or that it satisfies the hypotheses of Lemma \ref{lema3}, then Lemma \ref{lema2} or Lemma \ref{lema3}, respectively, guarantees that some element of $\operatorname{Span}(\mathcal{F})$ realizes exactly any prescribed configuration of $m\leq n$ (respectively, $m\leq n+1$) simple zeros; if, on the other hand, only the linear independence of $\mathcal{F}$, together with the existence of a constant-sign element, can be established, then Lemma \ref{lem1} guarantees at least $n$ simple zeros.
\end{enumerate}

The three proofs below differ only in the specific circles crossed (hence in the angles $\theta_{r}^{j}$ and in which $h_{1},h_{2},h_{3}$ appear), so we present each case in full and keep the lengthy Wronskian verifications in Appendix \ref{app.w}.
\section{Proof of Theorem  \ref{teoa}}\label{sec3}
\renewcommand{\proofname}{Proof of statement (a) of Theorem \ref{teoa}}
\begin{proof}
Note that for $\varepsilon =0$, the system (\ref{eq.it1}) is Hamiltonian with
\begin{eqnarray*}
H_{1}(x,y)=H_{2}(x,y)=H_{3}(x,y)=\frac{x^{2}+y^{2}}{2}.
\end{eqnarray*}
Also note that the system (\ref{eq.it1}) has a family of periodic orbits $L_{h}=L_{h}^{1}\cup L_{h}^{2}\cup L_{h}^{3}\cup L_{h}^{4}$ for $h\in (\frac{1}{2},\frac{9}{2})$, where
\begin{eqnarray*}
& &L^{1}_{h}=\left\{ \sqrt{2h}e^{it}, \pi -\arccos \left( \frac{3+2h}{4\sqrt{2h}}\right) \leq t\leq \pi+\arccos\left( \dfrac{3+2h}{4\sqrt{2h}}\right)\right\} ,\\
& &L^{2}_{h}=\left\{ \sqrt{2h}e^{it},\pi +\arccos \left( \frac{3+2h}{4\sqrt{2h}}\right) \leq t\leq 2\pi -\arccos \left( \dfrac{3+2h}{4\sqrt{2h}}\right)\right\} ,\\
& &L^{3}_{h}=\left\{  \sqrt{2h}e^{it}, -\arccos \left( \frac{3+2h}{4\sqrt{2h}}\right) \leq t\leq  \arccos\left( \dfrac{3+2h}{4\sqrt{2h}}\right)\right\} ,\\
& &L^{4}_{h}=\left\{  \sqrt{2h}e^{it}, \arccos \left( \frac{3+2h}{4\sqrt{2h}}\right) \leq t\leq \pi - \arccos\left( \dfrac{3+2h}{4\sqrt{2h}}\right)\right\} .\\
\end{eqnarray*}
Furthermore, each of the periodic orbits $L_{h}$ intersects the circle $|z+2|=1$ at the points 
\begin{eqnarray*}
   A_{1}(h)= \left(\frac{-3-2h}{4},\sqrt{2h-\left(\frac{3+2h}{4}\right)^{2}}\right)\,   \hbox{and}\,  A_{2}(h)=\left(\frac{-3-2h}{4},-\sqrt{2h-\left(\frac{3+2h}{4}\right)^{2}}\right),
\end{eqnarray*}
 and the circle $|z-2|=1$ at the points
 \begin{eqnarray*}
    \tilde{A}_{1}(h)=\left(\frac{3+2h}{4},\sqrt{2h-\left(\frac{3+2h}{4}\right)^{2}}\right)\,   \hbox{and}\,  \tilde{A}_{2}(h)=\left(\frac{3+2h}{4},-\sqrt{2h-\left(\frac{3+2h}{4}\right)^{2}}\right).
\end{eqnarray*}
Therefore, Assumptions \ref{Assumption1} and \ref{Assumption2} are satisfied. According to Theorem \ref{teomelnikov}, to estimate the number of limit cycles bifurcating from the center $\dot{z}=iz$, we need to calculate the number of zeros of the Melnikov function, which can be expressed as
\begin{align*}
M(h) = &\; \Re \left(\int_{L_{h}^{1}}i\overline{h_{1}}\,dz\right)+\Re \left(\int_{L_{h}^{2}}i\overline{h_{2}}\,dz\right) +\Re \left(\int_{L_{h}^{3}}i\overline{h_{3}}\,dz\right)+\Re \left(\int_{L_{h}^{4}}i\overline{h_{2}}\,dz\right).
\end{align*}
We will compute each of  the integrals that appear in the expression above. To simplify the notation, let $r=\sqrt{2h}$ and $\theta_{r}=\arccos\left( \frac{3+2h}{4\sqrt{2h}}\right)$.
Using Lemma \ref{lemaintegral}, as well as the trigonometric identities $\sin (n(\pi-x))=(-1)^{n+1}\sin (nx)$,  $\sin (n(\pi+x))=(-1)^{n}\sin (nx),$ and $\cos (n(\pi-x))=(-1)^{n}\cos (nx)=\cos (n(\pi+x))$, we have 
\begin{eqnarray*}
\Re \left(\int_{L_{h}^{1}}i\overline{h_{1}}dz\right)=2a_{1_{0}}r\sin (\theta_{r})-2a_{1_{1}}r^{2}\theta_{r}-2\sum_{j=2}^{n}\frac{r^{j+1}}{j-1}(-1)^{j-1}a_{1_{j}}\sin ((j-1)\theta_{r}).
\end{eqnarray*}
Applying Lemma \ref{lemaintegral}, together with the trigonometric identities 
$\sin (n(2\pi-x))=-\sin(nx)$, $\cos(n(2\pi-x))=\cos(nx)$, $\sin(n(\pi+x))=(-1)^{n}\sin(nx),$ and $\cos(n(\pi+x))=(-1)^{n}\cos(nx)$, we obtain
\begin{eqnarray*}
    \Re \left(\int_{L_{h}^{2}}i\overline{h_{2}}dz\right)&=&2b_{2_{0}}r\cos (\theta_{r})-a_{2_{1}}r^{2}(\pi -2\theta_{r})+2\sum_{j=1}^{\lfloor \frac{n - 1}{2} \rfloor}\frac{r^{2j+2}}{2j}a_{2_{2j+1}}\sin (2j\theta_{r})\\
    &&-\ 2\sum_{j=1}^{\lfloor \frac{n }{2} \rfloor}\frac{r^{2j+1}}{2j-1}b_{2_{2j}}\cos((2j-1)\theta_{r}).
\end{eqnarray*}

Now, using Lemma \ref{lemaintegral} and the trigonometric identities $\sin(-jx)=-\sin(jx)$ and $\cos(-jx)=\cos(jx)$, we obtain
\begin{eqnarray*}
     \Re \left(\int_{L_{h}^{3}}i\overline{h_{3}}dz\right)=-2a_{3_{0}}r\sin(\theta_{r})-2a_{3_{1}}r^{2}\theta_{r}-2\sum_{j=2}^{n}\frac{r^{j+1}}{j-1}a_{3_{j}}\sin((j-1)\theta_{r}).
\end{eqnarray*}

Finally, applying Lemma \ref{lemaintegral}, once again with the identities $\sin(n(\pi-x))=(-1)^{n+1}\sin (nx)$ and $\cos(n(\pi-x))=(-1)^{n}\cos (nx)$, we find
\begin{eqnarray*}
     \Re \left(\int_{L_{h}^{4}}i\overline{h_{2}}dz\right)&=&-2b_{20}r\cos (\theta_{r})-a_{21}r^{2}(\pi -2\theta_{r})+2\sum_{j=1}^{\lfloor \frac{n - 1}{2} \rfloor}\frac{r^{2j+2}}{2j}a_{2_{2j+1}}\sin (2j\theta_{r})\\
    &&+\ 2\sum_{J=1}^{\lfloor \frac{n }{2} \rfloor}\frac{r^{2j+1}}{2j-1}b_{2_{2j}}\cos((2j-1)\theta_{r}).
\end{eqnarray*}

Therefore
\begin{eqnarray*}
     M(r)=c_{0}r\sin (\theta_{r})+c_{1}r^{2}+d_{1}r^{2}\theta_{r}+\sum_{j=2}^{n}c_{j}r^{j+1}\sin((j-1)\theta_{r}),
\end{eqnarray*}
where
\begin{eqnarray*}
    c_{0} &=& 2a_{1_{0}} - 2a_{3_{0}}, \quad 
    c_{1} = -2\pi a_{2_{1}}, \quad 
    d_{1} = -2a_{1_{1}} + 4a_{2_{1}} - 2a_{3_{1}}, \\[6pt]
    c_{j} &=&
    \begin{cases}
        \frac{2}{j-1}(a_{1_{j}}-a_{3_{j}}), & j \text{ even },\\[8pt]
        \frac{2}{j-1}\left( -a_{1_{j}}+2a_{2_{j}} - a_{3_{j}}\right), & j \text{ odd }.
    \end{cases}
\end{eqnarray*}

In particular, when $n=4$,
\begin{eqnarray*}
    M(r)=c_{0} r\sin(\theta_{r}) + c_{1} r^{2} + d_{1} r^{2} \theta_{r}
 + c_{2} r^{3} \sin(\theta_{r}) + c_{3} r^{4} \sin(2\theta_{r}) + c_{4} r^{5}\sin(3\theta_{r})
\end{eqnarray*}
where
\[
\begin{aligned}
    c_{0} &= 2a_{1_{0}} - 2a_{3_{0}}, & \quad
    c_{1} &= -2\pi a_{2_{1}}, & \quad
    d_{1} &= -2a_{1_{1}} + 4a_{2_{1}} - 2a_{3_{1}}, \\
    c_{2} &= 2a_{1_{2}} - 2a_{3_{2}}, & \quad
    c_{3} &= -a_{1_{3}} + 2a_{2_{3}} - a_{3_{3}}, & \quad
    c_{4} &= \frac{2}{3} a_{1_{4}} - \frac{2}{3} a_{3_{4}}.
\end{aligned}
\]

Consider the ordered set of functions
\begin{eqnarray*}
    \mathcal{F}=\{ r\sin (\theta_{r}),r^3\sin(\theta_{r}),r^{4}\sin(2\theta_{r}),r^{5}\sin(3\theta_{r}),r^{2},r^{2}\theta_{r}\}.
\end{eqnarray*}
Note that $\theta_{r}$ fails to be differentiable at $r=1$ and $r=3$, where the level curve $|z|=r$ becomes tangent to the circle $|z+2|=1$; thus $\mathcal{F}$ is analytic on the open interval $(1,3)$, and Lemma \ref{lema3} must be invoked on a closed subinterval of it. Fix $[a,b]=[1.05,2.99]\subset(1,3)$. The Wronskians of $\mathcal{F}$ are computed in Appendix \ref{app.w} (Case 1): $W_{0},W_{1},W_{2},W_{3},W_{4}$ are nonzero throughout $(1,3)$, in particular on $[1.05,2.99]$, while $W_{5}$ has a single simple zero at $r\approx1.90012\in(1.05,2.99)$. Therefore, Lemma \ref{lema3}, applied on $[1.05,2.99]$, guarantees that $Z(\mathcal{F})=6$, and every configuration of $m\leq6$ simple zeros in $[1.05,2.99]$ is realized by some element of $\operatorname{Span}(\mathcal{F})$. Since the coefficients $c_{0},c_{1},d_{1},c_{2},c_{3},c_{4}$ can be chosen independently through the free parameters $a_{1_{j}},a_{2_{j}},a_{3_{j}}$, any such configuration can be realized by the perturbation. Hence, system (\ref{eq.it1}) has at least $6$ limit cycles bifurcating from the periodic orbits of the center $\dot{z}=iz$.
\end{proof}

\renewcommand{\proofname}{Proof of statement (b) of Theorem \ref{teoa}}
\begin{proof}
The following system
 \begin{eqnarray}\label{eq.ex.teoaa}
\begin{cases}
\dot{z}=iz+\varepsilon  h_{1}(z) \quad \mbox{if} \quad |z+2|<1,\\
\dot{z}=iz+\varepsilon h_{2}(z) \quad \mbox{if} \quad |z+2|>1 \quad \hbox{and }\quad |z-2|>1  \\
\dot{z}=iz+\varepsilon h_{3}(z) \quad \mbox{if} \quad |z-2|<1,
\end{cases}
\end{eqnarray}
has at least $6$ limit cycles for sufficiently small $\varepsilon $, where
\begin{eqnarray*}
    h_{1}(z)&=&\left( -\frac{1}{2}-\frac{0.0188988}{\pi}\right)z-0.02159415z^{4},\\
    h_{2}(z)&=&\left( \frac{-0.0188988}{2\pi}+2i\right)z+0.0713325z^{3},\\
    h_{3}(z)&=&0.2769615+0.391719z^{2}.
\end{eqnarray*}

Indeed, note that the first-order Melnikov function associated with this system is given by
\begin{align*}
M(r) =\;& 
    -0.553923\, r\sin(\theta_{r}) 
    +0.0188988\,r^{2} 
    + r^{2} \theta_{r} \\
    & -0.783438\,r^{3} \sin(\theta_{r}) 
    +0.142665\,r^{4} \sin(2\theta_{r}) 
    -0.0143961\,r^{5} \sin(3\theta_{r})
\end{align*}
and $M(r)$ has the six zeros
\begin{eqnarray*}
    r_{1}\approx 1.06, \quad r_{2}\approx 1.5, \quad r_{3}\approx 1.7, \quad r_{4} \approx 1.95, \quad r_{5}\approx 2.5, \quad r_{6}\approx  2.9.
\end{eqnarray*}

Moreover, differentiating $M(r),$ one finds that $M'(r)$ vanishes on $(1,3)$ at
\begin{eqnarray*}
    r\approx 1.17768, \quad 1.58774,\quad 1.84206, \quad 2.31754, \quad 2.78804,
\end{eqnarray*}
none of which coincides with $r_{1},\cdots , r_{6}$ above. Hence, $M'(r_{j})\neq 0$ for $j=1,\cdots ,6$, so each $r_{j}$ is a simple zero of $M(r)$. Therefore, for $\varepsilon $ sufficiently small, the system (\ref{eq.ex.teoaa}) has at least six limit cycles.
\begin{figure}[!htb]
\centering
\includegraphics[scale=0.68]{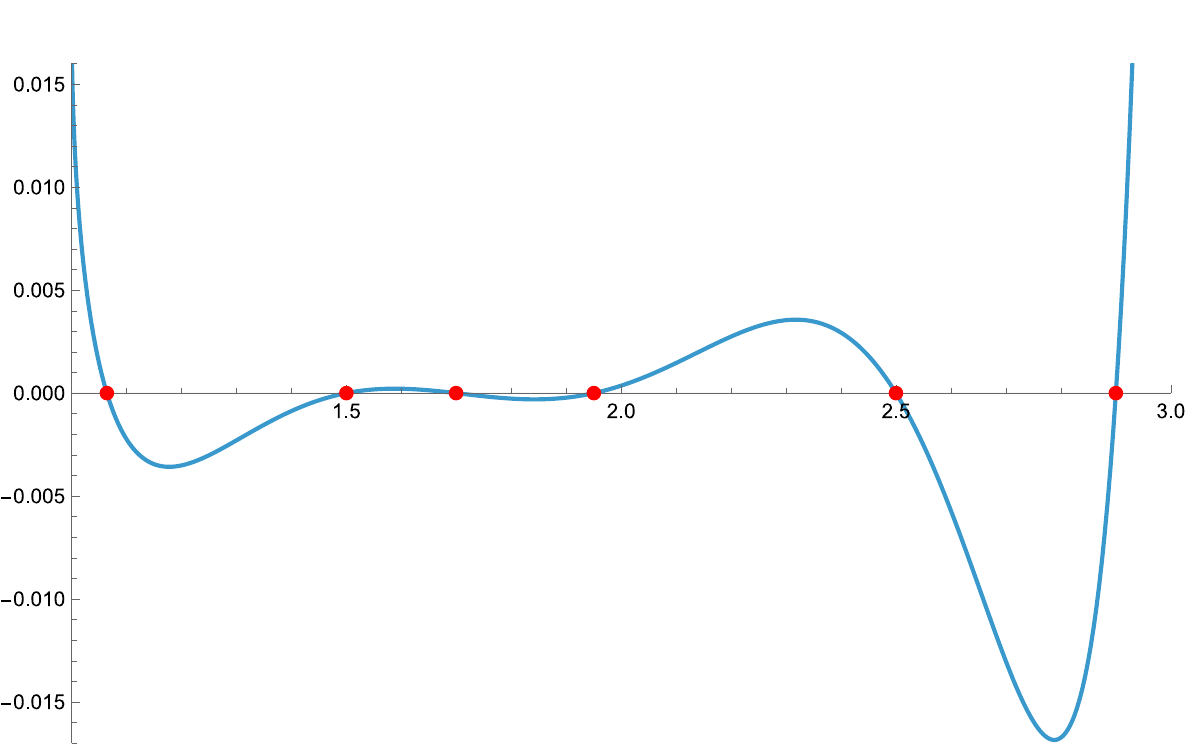}
\caption{Graph of the Melnikov function $M(r)$ associated with system (\ref{eq.ex.teoaa}). }
\label{imgteoa}
\end{figure}
\end{proof}

\section{Proof of Theorem  \ref{teob}}\label{sec4}
\renewcommand{\proofname}{Proof of statement (a) of Theorem \ref{teob}}
\begin{proof}
In this case, we have three distinct families of periodic orbits. The first one crosses only the circle  $|z-2|=\sqrt{2}$, denoted by $l_{h}=l_{h}^{3}\cup l_{h}^{2}$
for $h\in ( 3-2\sqrt{2},\frac{1}{2})$, where
\begin{eqnarray*}
& &l^{2}_{h}=\left\{ \sqrt{2h}e^{it}, \arccos \left( \frac{1+h}{2\sqrt{2h}}\right) \leq t\leq 2\pi-\arccos\left( \dfrac{1+h}{2\sqrt{2h}}\right)\right\} ,\\
& &l^{3}_{h}=\left\{ \sqrt{2h}e^{it},-\arccos \left( \frac{1+h}{2\sqrt{2h}}\right) \leq t\leq \arccos \left( \dfrac{1+h}{2\sqrt{2h}}\right)\right\}.
\end{eqnarray*}

The third family also crosses only the circle $|z-2|=\sqrt{2}$ and has the same expression as the previous one, but is defined on a different interval, with $h\in (\frac{9}{2},3+2\sqrt{2})$.
The second family intersects both circles $|z-2|=\sqrt{2}$ and $|z+2|=1$ and is given by  $L_{h}=L_{h}^{1}\cup L_{h}^{2}\cup L_{h}^{3}\cup L_{h}^{4}$ for $h\in (\frac{1}{2},\frac{9}{2})$, where 
\begin{eqnarray*}
& &L^{1}_{h}=\left\{ \sqrt{2h}e^{it}, \pi -\arccos \left( \frac{3+2h}{4\sqrt{2h}}\right) \leq t\leq \pi+\arccos\left( \dfrac{3+2h}{4\sqrt{2h}}\right)\right\} ,\\
& &L^{2}_{h}=\left\{ \sqrt{2h}e^{it},\pi +\arccos \left( \frac{3+2h}{4\sqrt{2h}}\right) \leq t\leq 2\pi -\arccos \left( \dfrac{1+h}{2\sqrt{2h}}\right)\right\} ,\\
& &L^{3}_{h}=\left\{  \sqrt{2h}e^{it}, -\arccos \left( \frac{1+h}{2\sqrt{2h}}\right) \leq t\leq  \arccos\left( \dfrac{1+h}{2\sqrt{2h}}\right)\right\} ,\\
& &L^{4}_{h}=\left\{  \sqrt{2h}e^{it}, \arccos \left( \frac{1+h}{2\sqrt{2h}}\right) \leq t\leq \pi - \arccos\left( \dfrac{3+2h}{4\sqrt{2h}}\right)\right\} .\\
\end{eqnarray*}

Note that the first and third families of periodic orbits intersect the curve $|z-2|=\sqrt{2}$ at two points
\begin{eqnarray*}
    a_{1}(h)=\left( \frac{h+1}{2},\sqrt{\frac{-h^{2}+6h-1}{4}}\right)\quad \hbox{and}\quad  a_{2}(h)=\left( \frac{h+1}{2},-\sqrt{\frac{-h^{2}+6h-1}{4}}\right).
\end{eqnarray*}

The second family of periodic orbits, in turn, intersects the circle $|z-2|=\sqrt{2}$ at the points 
\begin{eqnarray*}
   A_{1}(h)=\left( \frac{h+1}{2},\sqrt{\frac{-h^{2}+6h-1}{4}}\right) \quad \hbox{and}\quad A_{2}(h)=\left( \frac{h+1}{2},-\sqrt{\frac{-h^{2}+6h-1}{4}}\right),
\end{eqnarray*}
 and the circle $|z+2|=1$ at the points 
 \begin{eqnarray*}
     \tilde{A}_{1}(h)=\left(\frac{-3-2h}{4},\sqrt{2h-\left(\frac{3+2h}{4}\right)^{2}}\right) \quad \hbox{and}\quad \tilde{A}_{2}(h)=\left(\frac{-3-2h}{4},-\sqrt{2h-\left(\frac{3+2h}{4}\right)^{2}}\right)
 \end{eqnarray*}
 
Therefore, the Assumptions \ref{Assumption1} and \ref{Assumption2} of Theorem \ref{teomelnikov} are satisfied  by the family of orbits that cross both circles $|z-2|=\sqrt{2}$ and $|z+2|=1$, whereas the families of orbits that cross only $|z-2|=\sqrt{2}$ satisfy Assumptions $(1)$ and $(2)$ of  Remark \ref{obs.1}.
For the families that cross only $|z-2|=\sqrt{2}$, it follows that the Melnikov function is given by:
\begin{eqnarray*}
    m(h)=\Re \left( \int_{l_{h}^{2}}i\overline{h_{2}(z)}dz\right)+\Re \left( \int_{l_{h}^{3}}i\overline{h_{3}(z)}dz\right).
\end{eqnarray*}

For simplicity, we define $r=\sqrt{2h}$, $\theta_{r}^{1}=\arccos \left( \frac{1+h}{2\sqrt{2h}}\right)=\arccos \left( \frac{1+r^{2}/2}{2r}\right)$ and $\theta_{r}^{2}=\arccos \left( \frac{3+2h}{4\sqrt{2h}}\right)=\arccos \left( \frac{3+r^{2}}{4r}\right)$.
By applying Lemma \ref{lemaintegral} along with the trigonometric identities $\cos(-nx)=\cos(nx)=\cos(n(2\pi-x))$ and $\sin(-nx)=-\sin(nx)=\sin (n(2\pi -x))$, we deduce that 
\begin{eqnarray*}
    \Re \left( \int_{l_{h}^{2}}i\overline{h_{2}(z)}dz\right)=2a_{2_{0}}r\sin (\theta_{r}^{1})-a_{2_{1}}r^{2}(2\pi -2\theta_{r}^{1})+2\sum_{j=2}^{n}\frac{r^{j+1}}{j-1}a_{2_{j}}\sin((j-1)\theta_{r}^{1})
\end{eqnarray*}
and
\begin{eqnarray*}
    \Re \left( \int_{l_{h}^{3}}i\overline{h_{3}(z)}dz\right)=-2a_{3_{0}}r\sin(\theta_{r}^{1})-2a_{3_{1}}r^{2}\theta_{r}^{1}-2\sum_{j=2}^{n}\frac{r^{j+1}}{j-1}a_{3_{j}}\sin((j-1)\theta_{r}^{1}).
\end{eqnarray*}

Therefore
\begin{eqnarray*}
    m(r)=c_{0}r\sin(\theta_{r}^{1})+c_{1}r^{2}\theta_{r}^{1}+d_{1}r^{2}+\sum_{j=2}^{n}c_{j}r^{j+1}\sin ((j-1)\theta_{r}^{1}),
\end{eqnarray*}
where
\begin{align*}
c_0 &= 2a_{2_0} - 2a_{3_0}, & c_1 &= 2a_{2_1}-2a_{3_1}, \\
d_1 &= -2\pi a_{2_1}, & c_j &= 2\left(\frac{a_{2_j} - a_{3_j}}{j - 1}\right)
\end{align*}

In the specific case when $n=4$,
\begin{eqnarray*}
    m(r) &=&c_0 r \sin(\theta_r^1) + c_1 r^2 \theta_r^1 + d_1 r^2  
+ c_2 r^3 \sin(\theta_r^1) + c_3 r^4 \sin(2 \theta_r^1) + c_4 r^5 \sin(3 \theta_r^1).
\end{eqnarray*}

Consider the ordered set of functions
\begin{eqnarray*}
    \mathcal{F}=\{ r\sin (\theta_{r}^{1}),r^3\sin(\theta_{r}^{1}),r^{4}\sin(2\theta_{r}^{1}),r^{5}\sin(3\theta_{r}^{1}),r^{2},r^{2}\theta_{r}^{1}\}.
\end{eqnarray*}

Note that $\theta_{r}^{1}$ fails to be differentiable at $r=\sqrt{2(3-2\sqrt{2})}$ and $r=\sqrt{2(3+2\sqrt{2})}$, the endpoints of these two intervals, where the level curve becomes tangent to $|z-2|=\sqrt{2}$; so $\mathcal{F}$ is analytic only on the open intervals, and Lemma \ref{lema2} must be applied on closed subintervals of them. Fix $[a_{1},b_{1}]=[0.589,0.99]\subset\left(\sqrt{2(3-2\sqrt{2})},1\right)$ and $[a_{2},b_{2}]=[3.01,3.41]\subset\left(3,\sqrt{2(3+2\sqrt{2})}\right)$. By Appendix \ref{app.w} (Case 2), all six Wronskians $W_{0},\cdots , W_{5}$ are nonzero throughout $\left(\sqrt{2(3-2\sqrt{2})},1\right)$ and $\left(3,\sqrt{2(3+2\sqrt{2})} \right)$, in particular on $[a_{1},b_{1}]$ and $[a_{2},b_{2}]$, so $\mathcal{F}$ is an ECT-system on each of these closed intervals. Therefore, Lemma \ref{lema2} ensures that the number of simple zeros of any element of $\operatorname{Span}(\mathcal{F})$ on $[a_{1},b_{1}]$ (resp. $[a_{2},b_{2}]$) does not exceed $5$, and every configuration of $m\leq 5$ prescribed simple zeros in $[a_{1},b_{1}]$ (resp. $[a_{2},b_{2}]$) is realized by some element of $\operatorname{Span}(\mathcal{F})$. Since the coefficients $c_{0},c_{1},d_{1},c_{2},c_{3},c_{4}$ are linearly independent functions of the free perturbation parameters $a_{2_{j}},a_{3_{j}}$, such a configuration can indeed be realized by a suitable choice of $h_{2},h_{3}$. Hence, system (\ref{eq.it2}) has at least $5$ limit cycles bifurcating from the center $\dot{z}=iz$, which intersect only the circle $|z-2|=\sqrt{2}$.
Now, for the family of orbits that crosses both circles, the expression of the Melnikov function is given by
\begin{align*}
M(h) = &\; \Re \left(\int_{L_{h}^{1}}i\overline{h_{1}}\,dz\right)+\Re \left(\int_{L_{h}^{2}}i\overline{h_{2}}\,dz\right) +\Re \left(\int_{L_{h}^{3}}i\overline{h_{3}}\,dz\right)+\Re \left(\int_{L_{h}^{4}}i\overline{h_{2}}\,dz\right).
\end{align*}

Applying Lemma \ref{lemaintegral}, we have that
\begin{eqnarray*}
    \Re \left(\int_{L_{h}^{1}}i\overline{h_{1}}dz\right)&=&2a_{1_0}r\sin (\theta_{r}^{2})-2a_{1_1}r^{2}\theta_{r}^{2}-2\sum_{j=2}^{n}\frac{r^{j+1}}{j-1}(-1)^{j-1}a_{1_j}\sin ((j-1)\theta_{r}^{2}),\\
    \Re \left(\int_{L_{h}^{2}}i\overline{h_{2}}dz\right)&=&a_{2_{0}}r\left( \sin (\theta_{r}^{1})-\sin(\theta_{r}^{2})\right) +b_{2_{0}}r\left( \cos(\theta_{r}^{1})+\cos(\theta_{r}^{2})\right)+a_{2_{1}}r^{2}\left( -\pi +\theta_{r}^{1}+\theta_{r}^{2}\right) \\
    &+&\sum_{j=2}^{n}\frac{a_{2_{j}}r^{j+1}}{j-1}\left( \sin((j-1)\theta_{r}^{1})+(-1)^{j-1}\sin((j-1)\theta_{r}^{2})\right)\\
    &+&\sum_{j=2}^{n}\frac{b_{2_{j}}r^{j+1}}{j-1}\left( -\cos((j-1)\theta_{r}^{1})+(-1)^{j-1}\cos((j-1)\theta_{r}^{2})\right)\\
    \Re \left(\int_{L_{h}^{3}}i\overline{h_{3}}dz\right)&=&-2a_{3_{0}}r\sin(\theta_{r}^{1})-2a_{3_{1}}r^{2}\theta_{r}^{1}-2\sum_{j=2}^{n}\frac{r^{j+1}}{j-1}a_{3_{j}}\sin((j-1)\theta_{r}^{1})\\
    \Re \left(\int_{L_{h}^{4}}i\overline{h_{2}}dz\right)&=&a_{2_{0}}r\left(\sin (\theta_{r}^{1})-\sin(\theta_{r}^{2}) \right)-b_{2_{0}}r\left( \cos (\theta_{r}^{1})+\cos(\theta_{r}^{2})\right)+a_{2_{1}}r^{2}\left( -\pi +\theta_{r}^{1}+\theta_{r}^{2} \right)\\
   &+& \sum_{j=2}^{n}\frac{a_{2_{j}}r^{j+1}}{j-1}\left( \sin((j-1)\theta_{r}^{1})+(-1)^{j-1}\sin((j-1)\theta_{r}^{2})\right)\\
    &-&\sum_{j=2}^{n}\frac{b_{2_{j}}r^{j+1}}{j-1}\left( -\cos((j-1)\theta_{r}^{1})+(-1)^{j-1}\cos((j-1)\theta_{r}^{2})\right)
\end{eqnarray*}

Therefore,
\begin{eqnarray*}
    M(r)&=&(2a_{2_{0}}-2a_{3_{0}})r\sin (\theta_{r}^{1})+(2a_{1_{0}}-2a_{2_{0}})r\sin (\theta_{r}^{2})+(2a_{2_{1}}-2a_{3_{1}})r^{2}\theta_{r}^{1}\\
    &+&(2a_{2_{1}}-2a_{1_{1}})r^{2}\theta_{r}^{2}+(-2\pi a_{2_1})r^{2}+\sum_{j=2}^{n}\left( \frac{2a_{2_{j}}-2a_{3_{j}}}{j-1}\right) r^{j+1}\sin ((j-1)\theta_{r}^{1})\\
    &+&\sum_{j=2}^{n}\left( \frac{(-1)^{j-1}(2a_{2_{j}}-2a_{1_{j}})}{j-1}\right) r^{j+1}\sin ((j-1)\theta_{r}^{2}).
\end{eqnarray*}

In particular, when $n=4$, we have
\begin{eqnarray*}
    M(r)=\sum_{j=1}^{11}A_{j}f_{j}(r)
\end{eqnarray*}
where
\[
\begin{aligned}
f_{1}(r) &= r\sin(\theta_{r}^{1}), &\quad 
f_{2}(r) &= r\sin(\theta_{r}^{2}), &\quad 
f_{3}(r) &= r^{2}\theta_{r}^{1}, \\
f_{4}(r) &= r^{2}\theta_{r}^{2}, &\quad 
f_{5}(r) &= r^{2}, &\quad 
f_{6}(r) &= r^{3}\sin(\theta_{r}^{1}), \\
f_{7}(r) &= r^{3}\sin(\theta_{r}^{2}), &\quad 
f_{8}(r) &= r^{4}\sin(2\theta_{r}^{1}), &\quad 
f_{9}(r) &= r^{4}\sin(2\theta_{r}^{2}), \\
f_{10}(r) &= r^{5}\sin(3\theta_{r}^{1}), &\quad 
f_{11}(r) &= r^{5}\sin(3\theta_{r}^{2}), &
\end{aligned}
\]
and 
\[
\begin{aligned}
A_{1} &= (2a_{2_{0}} - 2a_{3_{0}}), &\quad 
A_{2} &= (2a_{1_{0}} - 2a_{2_{0}}), &\quad 
A_{3} &= (2a_{2_{1}} - 2a_{3_{1}}), \\
A_{4} &= (2a_{2_{1}} - 2a_{1_{1}}), &\quad 
A_{5} &= (-2\pi a_{2_1}), &\quad 
A_{6} &= (2a_{2_{2}} - 2a_{3_{2}}), \\
A_{7} &= -(2a_{2_{2}} - 2a_{1_{2}}), &\quad 
A_{8} &= \frac{(2a_{2_{3}} - 2a_{3_{3}})}{2}, &\quad 
A_{9} &= \frac{(2a_{2_{3}} - 2a_{1_{3}})}{2}, \\
A_{10} &= \frac{(2a_{2_{4}} - 2a_{3_{4}})}{3}, &\quad 
A_{11} &= -\frac{(2a_{2_{4}} - 2a_{1_{4}})}{3}. &
\end{aligned}
\]

We have $W(f_{1},\cdots ,f_{11})(\frac{5}{2})\approx 9.979093 \times10^{25}\neq 0$. Therefore, the functions $f_{1},\cdots ,f_{11}$ are linearly independent. In addition, $f_{3}$ maintains a constant sign over the interval $(1 ,3)$, and
the coefficients  $A_{1},\cdots ,A_{11}$ can be chosen arbitrarily: as displayed above, each $A_{j}$ is an explicit linear function of the free real parameters $a_{1_{j}},a_{2_{j}},a_{3_{j}}$ of $h_{1},h_{2},h_{3}$. Consequently, by Lemma \ref{lem1}(b), it follows that one can select these coefficients so that $M(r)$ has at least $10$ simple zeros. Hence, system (\ref{eq.it2}) has at least $10$ limit cycles bifurcating from the center $\dot{z}=iz$, which intersect both $|z+2|=1$ and $|z-2|=\sqrt{2}$.
\end{proof}

\renewcommand{\proofname}{Proof of statement (b) of Theorem \ref{teob}}
\begin{proof}
    The following system
 \begin{eqnarray}\label{eq.ex.teobb}
\begin{cases}
\dot{z}=iz+\varepsilon h_{1}(z) \quad \mbox{if} \quad |z+2|<1,\\
\dot{z}=iz+\varepsilon  h_{2}(z) \quad \mbox{if} \quad |z+2|>1 \quad \hbox{and }\quad |z-2|>\sqrt{2}  \\
\dot{z}=iz+\varepsilon  h_{3}(z) \quad \mbox{if} \quad |z-2|<\sqrt{2},
\end{cases}
\end{eqnarray}
has at least $16$ limit cycles for sufficiently small $\varepsilon $, where
\begin{eqnarray*}
    h_{1}(z)&=& -0.0504822672096+0.4539003764375z+0.0778357303430z^2\\
     &\quad &+0.4231906218954z^3-0.0000601091140z^{4},\\
    h_{2}(z)&=&-0.1422087015474-0.0008344860991z-0.5z^{2}\\
    &\quad &+0.1984483885666z^{3}-0.0265841713573z^{4},\\
    h_{3}(z)&=&-0.4560166059445z
\end{eqnarray*}

Indeed, as we have seen earlier, when $\varepsilon =0$, the system has three families of periodic orbits. Thus, the first-order Melnikov function associated with (\ref{eq.ex.teobb}) and corresponding to the first family is given by
\begin{align*}
m_{1}(r) =\;& 
    -0.28441740309482\, r\sin(\theta_{r}^{1}) 
    + 0.0052432307972058\,r^{2} 
    + 0.91036423969070\, r^{2} \theta_{r}^{1} \\
    &-\,r^{3} \sin(\theta_{r}^{1}) 
    + 0.19844838856661\,r^{4} \sin(2\theta_{r}^{1}) 
   -0.017722780904860\,r^{5} \sin(3\theta_{r}^{1}),
\end{align*}
and $m_{1}(r)$ has  three zeros
\begin{eqnarray*}
    r_{1}\approx  0.599630, \quad r_{2}\approx 0.634593, \quad r_{3}\approx  0.961079.
\end{eqnarray*}

Differentiating $m_{1}(r)$, we find that $m_{1}'(r)$ vanishes on $\left(\sqrt{2(3-2\sqrt{2})},1 \right)$  at $r\approx  0.614834$ and $r\approx 0.827937$, distinct from $r_{1},r_{2},r_{3}$ above. Hence, $m'_{1}(r_{j})\neq 0$ for $j=1,2,3$, so all three zeros of $m_{1}$ are simple.
The first-order Melnikov function associated with system (\ref{eq.ex.teobb}) and corresponding to the third family of periodic orbits, $m_{2}$, has the same expression as $m_{1}$. However, it is defined on the interval $\left(3,\sqrt{2(3+2\sqrt{2})} \right)$. In this interval, it has three zeros
\begin{eqnarray*}
    r_{1}\approx  3.013205, \quad r_{2}\approx 3.138776, \quad r_{3}\approx  3.401237.
\end{eqnarray*}

Likewise, on this interval $m'_{2}$ vanishes at $r\approx 3.075700$ and $r\approx 3.332528$, distinct from $r_{1},r_{2},r_{3}$ above. Hence, all three zeros of $m_{2}$ are simple.
\begin{figure}[!htb]
\centering
\begin{minipage}{0.48\textwidth}
    \centering
    \includegraphics[width=\textwidth]{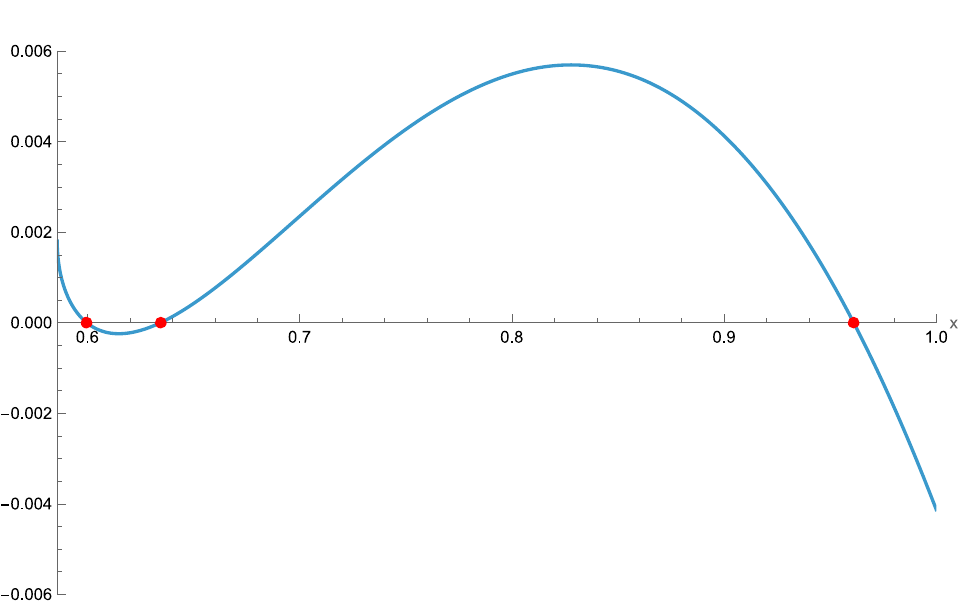}
\end{minipage}
\hspace{-0.3cm}
\begin{minipage}{0.48\textwidth}
    \centering
    \includegraphics[width=\textwidth]{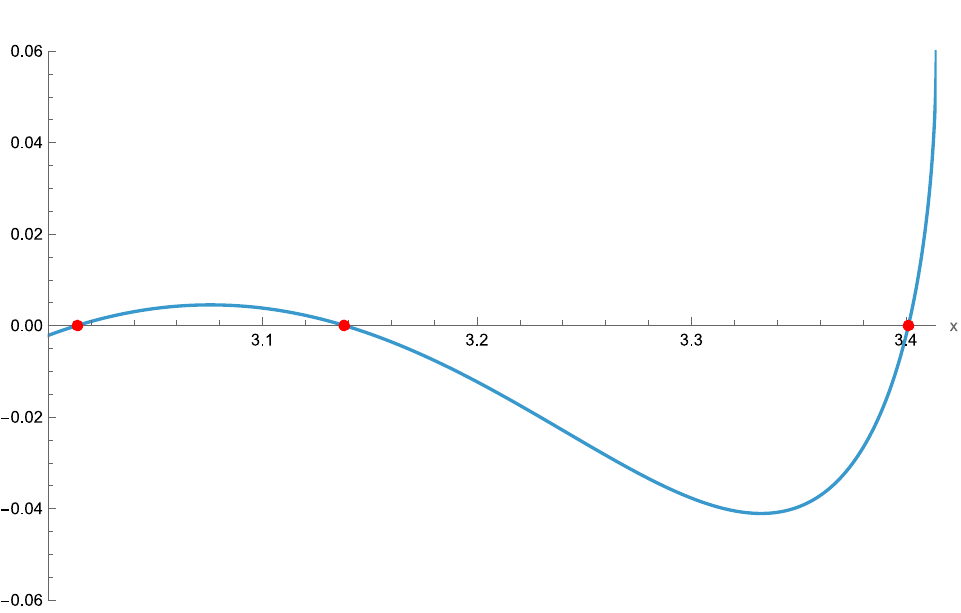}
\end{minipage}
\caption{Graph of the Melnikov functions $m_1(r)$ and $m_{2}(r)$ associated with system (\ref{eq.ex.teobb}).}
\end{figure}
Finally, the first-order Melnikov function corresponding to the second family is given by
\begin{align*}
M(r) =\;& 
    -0.28441740309482\, r\sin(\theta_{r}^{1}) +  0.18345286867571\, r\sin(\theta_{r}^{2}) + 0.91036423969070 r^{2} \theta_{r}^{1}\\
   & -0.90946972507325r^{2} \theta_{r}^{2}+ 0.0052432307972058\,r^{2} 
     \\
    &-\,r^{3} \sin(\theta_{r}^{1}) + 1.15567146068591\,r^{3} \sin(\theta_{r}^{2})\\
    &+0.1984483885666\,r^{4} \sin(2\theta_{r}^{1}) - 0.22474223332883\,r^{4} \sin(2\theta_{r}^{2})\\
   & -0.017722780904860\,r^{5} \sin(3\theta_{r}^{1})+ 0.017682708162225\,r^{5} \sin(3\theta_{r}^{2})
\end{align*}
and it has ten zeros
\begin{eqnarray*}
    r_{1}& \approx & 1.304571, \quad r_{2}\approx 1.717150, \quad r_{3}\approx 2.040756, \quad r_{4} \approx  2.196555, \quad r_{5}\approx 2.329536,
    \\ r_{6}&\approx &2.409378, \quad r_{7}\approx 2.473021, \quad r_{8}\approx 2.684530, \quad r_{9} \approx 2.756957, \quad r_{10}\approx 2.851674. 
\end{eqnarray*}

Moreover, differentiating $M(r)$, one finds that $M'(r)$ vanishes on $(1,3)$  at the nine points 
\begin{eqnarray*}
    r& \approx &  1.391439,\;  1.803607,\; 2.095536,\; 2.250281,\; 2.367641,\\
     r& \approx &2.446814,\; 2.618856,\; 2.729326,\; 2.828435,
\end{eqnarray*}
none of which coincides with  $r_{1},\cdots , r_{10}$ above. Hence, $M'(r_{j})\neq0 $ for $j=1,\cdots, 10$, so all ten zeros of $M(r)$ are simple.
Therefore, for $\varepsilon $ sufficiently small, the system (\ref{eq.ex.teobb}) has at least sixteen limit cycles.
\begin{figure}[!htb]
\centering
\includegraphics[scale=0.68]{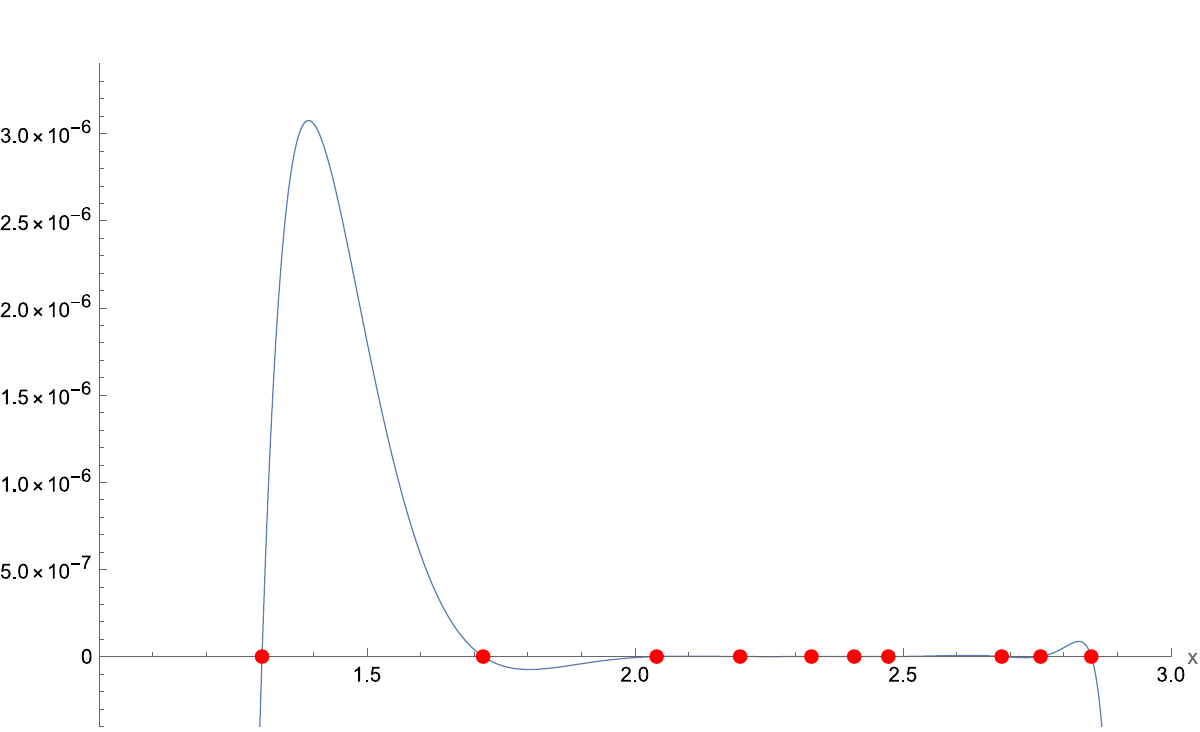}
\caption{Graph of the Melnikov function $M(r)$ associated with system (\ref{eq.ex.teobb}). }
\label{imgteob}
\end{figure}
\end{proof}

\section{Proof of Theorem  \ref{teoc}}\label{sec5}

\renewcommand{\proofname}{Proof of statement (a) of Theorem \ref{teoc}}
\begin{proof}
Note that we have three families of periodic orbits. The first family crosses only the circle $|z+2|=1$ and is denoted by $l_{h}=l_{h}^{1}\cup l_{h}^{2}$ for $h\in (\frac{1}{2},3-\sqrt{5})$, where
\begin{eqnarray*}
& &l^{1}_{h}=\left\{ \sqrt{2h}e^{it},\pi- \arccos \left( \frac{3+2h}{4\sqrt{2h}}\right) \leq t\leq \pi+\arccos\left( \dfrac{3+2h}{4\sqrt{2h}}\right)\right\} ,\\
& &l^{2}_{h}=\left\{ \sqrt{2h}e^{it},\pi+\arccos \left( \frac{3+2h}{4\sqrt{2h}}\right) \leq t\leq 3\pi-\arccos \left( \dfrac{3+2h}{4\sqrt{2h}}\right)\right\}.
\end{eqnarray*}

The second family intersects the circles $|z+2|=1$ and $|z-2-i|=1$ and is given by $L_{h}=L_{h}^{1}\cup L_{h}^{2}\cup L_{h}^{3}\cup L_{h}^{4}$, for $h\in (3-\sqrt{5},\frac{9}{2})$, where
\begin{eqnarray*}
& &L^{1}_{h}=\left\{ \sqrt{2h}e^{it}, \pi -\arccos \left( \frac{3+2h}{4\sqrt{2h}}\right) \leq t\leq \pi+\arccos\left( \dfrac{3+2h}{4\sqrt{2h}}\right)\right\} ,\\
& &L^{2}_{h}=\left\{ \sqrt{2h}e^{it},\pi +\arccos \left( \frac{3+2h}{4\sqrt{2h}}\right) \leq t\leq 2\pi+\arctan\left( \frac{1}{2}\right)-\arccos \left( \frac{2+h}{\sqrt{10h}}\right)\right\} ,\\
& &L^{3}_{h}=\left\{  \sqrt{2h}e^{it}, 2\pi +\arctan\left( \frac{1}{2}\right)-\arccos \left( \frac{2+h}{\sqrt{10h}}\right) \leq t\leq 2\pi +\arctan \left( \frac{1}{2}\right) + \arccos\left( \frac{2+h}{\sqrt{10h}}\right)\right\} ,\\
& &L^{4}_{h}=\left\{  \sqrt{2h}e^{it},2\pi+\arctan \left( \frac{1}{2}\right)+\arccos \left( \frac{2+h}{\sqrt{10h}}\right) \leq t\leq 3\pi - \arccos\left( \dfrac{3+2h}{4\sqrt{2h}}\right)\right\} .\\
\end{eqnarray*}

The third family intersects only the circle $|z-2-i|=1$ and is given by $m_{h}=m_{h}^{2}\cup m_{h}^{3}$, for $h\in (\frac{9}{2},3+\sqrt{5})$, where
\begin{eqnarray*}
& &m^{3}_{h}=\left\{ \sqrt{2h}e^{it},\arctan \left( \frac{1}{2}\right)- \arccos \left( \frac{2+h}{\sqrt{10h}}\right) \leq t\leq \arctan \left( \frac{1}{2}\right)+\arccos\left( \dfrac{2+h}{\sqrt{10h}}\right)\right\} ,\\
& &m^{2}_{h}=\left\{ \sqrt{2h}e^{it},\arctan \left( \frac{1}{2}\right)+\arccos \left( \frac{2+h}{\sqrt{10h}}\right) \leq t\leq 2\pi+\arctan \left( \frac{1}{2}\right)- \arccos \left( \frac{2+h}{\sqrt{10h}}\right)\right\}.
\end{eqnarray*}

Consider the points
\begin{eqnarray*}
    A_{1}(h)&=&\left(\frac{-3-2h}{4},\frac{\sqrt{-4h^{2}+20h-9}}{4} \right),\, A_{2}(h)=\left(\frac{-3-2h}{4},-\frac{\sqrt{-4h^{2}+20h-9}}{4} \right)\\
    \tilde{A}_{1}(h)&=&\left(\frac{2h+4-\sqrt{-h^{2}+6h-4}}{5},\frac{h+2+2\sqrt{-h^{2}+6h-4}}{5} \right), \\ \tilde{A}_{2}(h)&=&\left(\frac{2h+4+\sqrt{-h^{2}+6h-4}}{5},\frac{h+2-2\sqrt{-h^{2}+6h-4}}{5} \right).
\end{eqnarray*}

Observe that the first family $l_{h}$ intersects $|z+2|=1$ at the points $A_{1}(h)$ and $A_{2}(h)$, the second family $L_{h}$ intersects $|z+2|=1$ at the points $A_{1}(h)$ and $A_{2}(h)$ and $|z-2-i|=1$ at the points $\tilde{A}_{1}(h)$ and $\tilde{A}_{2}(h)$, and finally the third family $m_{h}$ intersects $|z-2-i|=1$ at the points $\tilde{A}_{1}(h)$ and  $\tilde{A}_{2}(h)$.
We then have that Assumptions \ref{Assumption1} and \ref{Assumption2} of Theorem \ref{teomelnikov} are satisfied for the family $L_{h}$, while Assumptions (1) and (2) of Remark \ref{obs.1} are satisfied for the families $l_{h}$ and $m_{h}$.
To estimate the number of limit cycles that bifurcate from the center $\dot{z}=iz$, it is necessary to analyze the Melnikov functions corresponding to each of the families. We have that, for each family, the corresponding expressions of the Melnikov function are as follows:
For the family that crosses only the circle $|z+2|=1$, the expression of the Melnikov function is given by:
\begin{eqnarray*}
    M_{1}(h)=\Re \left( \int_{l_{h}^{1}}i\overline{h_{1}(z)}dz\right)+\Re \left( \int_{l_{h}^{2}}i\overline{h_{2}(z)}dz\right).
\end{eqnarray*}

For the family that crosses both circles, the Melnikov function is given by:
\[
M_{2}(h)=\Re \left(\int_{L_{h}^{1}}i\overline{h_{1}}dz\right)+\Re \left(\int_{L_{h}^{2}}i\overline{h_{2}}dz\right)\]\[+\Re \left(\int_{L_{h}^{3}}i\overline{h_{3}}dz\right)+\Re \left(\int_{L_{h}^{4}}i\overline{h_{2}}dz\right).
\]

Finally, for the family that crosses only the circle $|z-2-i|=1$, the expression of the Melnikov function is given by:
\begin{eqnarray*}
    M_{3}(h)=\Re \left( \int_{m_{h}^{2}}i\overline{h_{2}(z)}dz\right)+\Re \left( \int_{m_{h}^{3}}i\overline{h_{3}(z)}dz\right).
\end{eqnarray*}

We will analyze each of the functions separately. We will begin with $M_{1}$, and we will also need to compute each of the integrals individually.  In order to simplify the notation, we adopt the expressions $r=\sqrt{2h}$, $\theta_{r}^{1}=\arccos\left(\frac{3+r^{2}}{4r} \right)$.
By applying Lemma \ref{lemaintegral} and the identities $\sin\left(n(\pi-x) \right)=(-1)^{n+1}\sin \left( nx\right)$, $\sin(n(\pi+x))=(-1)^{n}\sin(nx)$,  $\cos(n(\pi-x))=(-1)^{n}\cos(nx)=\cos(n(\pi+x))$, $\sin (n(3\pi-x))=(-1)^{n+1}\sin (nx)$ and $ \cos (n(3\pi-x))=(-1)^{n}\cos (nx)$, we have 
\begin{eqnarray*}
    \Re \left( \int_{l_{h}^{1}}i\overline{h_{1}(z)}dz\right)&=&2a_{1_{0}}r\sin (\theta_{r}^{1})-2a_{1_{1}}r^{2}\theta_{r}^{1}-2\sum_{j=2}^{n}\frac{r^{j+1}}{j-1}(-1)^{j-1}a_{1_{j}}\sin ((j-1)\theta_{r}^{1})
\end{eqnarray*}    
and
\begin{eqnarray*}
    \Re \left( \int_{l_{h}^{2}}i\overline{h_{2}(z)}dz\right)&=&-2a_{2_{0}}r\sin (\theta_{r}^{1})-2a_{2_{1}}r^{2}(\pi -\theta_{r}^{1})+2\sum_{j=2}^{n}\frac{r^{j+1}}{j-1}(-1)^{j-1}a_{2_{j}}\sin ((j-1)\theta_{r}^{1}).
\end{eqnarray*}

Consequently
\begin{eqnarray*}
 M_{1}(r)=c_{0}r\sin(\theta_{r}^{1})+c_{1}r^{2}\theta_{r}^{1}+d_{1}r^{2}+\sum_{j=2}^{n}c_{j}r^{j+1}\sin ((j-1)\theta_{r}^{1}),
\end{eqnarray*}
where
\begin{align*}
c_0 &= 2a_{1_{0}} - 2a_{2_{0}}, & c_1 &= 2a_{2_1}-2a_{1_1}, \\
d_1 &= -2\pi a_{2_{1}}, & c_j &= 2(-1)^{j-1}\left(\frac{a_{2_j} - a_{1_j}}{j - 1}\right)
\end{align*}

Particularly when $n=4$,
\begin{eqnarray*}
    M_{1}(r) &=&c_0 r \sin(\theta_r^1) + c_1 r^2 \theta_r^1 + d_1 r^2  
+ c_2 r^3 \sin(\theta_r^1) + c_3 r^4 \sin(2 \theta_r^1) + c_4 r^5 \sin(3 \theta_r^1).
\end{eqnarray*}

Consider the ordered set of functions
\begin{eqnarray*}
    \mathcal{F}=\{ r\sin (\theta_{r}^{1}),r^3\sin(\theta_{r}^{1}),r^{4}\sin(2\theta_{r}^{1}),r^{5}\sin(3\theta_{r}^{1}),r^{2},r^{2}\theta_{r}^{1}\},
\end{eqnarray*}
and note that $\theta_r^1$ is built from the same circle $|z+2|=1$ as in the proof of Theorem \ref{teoa}, so it fails to be differentiable at $r=1$, and $\mathcal F$ has exactly the same Wronskians computed in Appendix \ref{app.w} (Case 1). Since $\sqrt{2(3-\sqrt5)}\approx1.236<1.90012$, the interval $\left(1,\sqrt{2(3-\sqrt5)}\right)$ --- corresponding to $h\in(\frac12,3-\sqrt5)$ --- is strictly contained in $(1,3)$ and excludes the point $r\approx1.90012$ where $W_5$ vanishes. Fix a closed subinterval $[a,b]=[1.05,1.2]\subset\left(1,\sqrt{2(3-\sqrt5)}\right)$, on which $\mathcal{F}$ is analytic and all six Wronskians are nonzero. Therefore, Lemma \ref{lema2}, applied on $[1.05,1.2]$, guarantees that the number of simple zeros of any element of $\operatorname{Span}(\mathcal{F})$ on $[1.05,1.2]$ does not exceed $5$, and every configuration of $m\leq5$ prescribed simple zeros in $[1.05,1.2]$ is realized by some element of $\operatorname{Span}(\mathcal{F})$. Moreover, since $c_{0},c_{1},d_{1},c_{2},c_{3},c_{4}$ depend linearly and independently on the free parameters of $h_{1},h_{2}$, this configuration of five simple zeros, placed in $[1.05,1.2]\subset\left(1,\sqrt{2(3-\sqrt5)}\right)$, can indeed be achieved through a suitable choice of the perturbation. As a result, system (\ref{eq.it3}) has at least $5$ limit cycles bifurcating from the center $\dot{z}=iz$, which intersect  only the circle $|z+2|=1$.

With the analysis of $M_{1}$ complete, we now turn our attention to $M_{2}$. In addition to $\theta_{r}^{1}$, we define
\begin{eqnarray*}
    \theta_{r}^{2}=\arctan\left(\frac{1}{2}\right) - \arccos\left(\frac{2+h}{\sqrt{10h}}\right) \quad \hbox{and}\quad \theta_{r}^{3}=\arctan\left(\frac{1}{2}\right) + \arccos\left(\frac{2+h}{\sqrt{10h}}\right).
\end{eqnarray*}
Then, applying Lemma \ref{lemaintegral}, we have that 
\begin{eqnarray*}
   \Re \left( \int_{L_{h}^{1}}i\overline{h_{1}(z)}dz\right)&=&2a_{1_{0}}r\sin (\theta_{r}^{1})-2a_{1_{1}}r^{2}\theta_{r}^{1}-2\sum_{j=2}^{n}\frac{r^{j+1}}{j-1}(-1)^{j-1}a_{1_{j}}\sin ((j-1)\theta_{r}^{1}) 
 \end{eqnarray*} 
 \begin{eqnarray*}
   \Re \left( \int_{L_{h}^{2}}i\overline{h_{2}(z)}dz\right)&=&-a_{2_{0}}r(\sin (\theta_{r}^{2})+\sin (\theta_{r}^{1}))+b_{2_{0}}r(\cos (\theta_{r}^{2})+\cos (\theta_{r}^{1}))+a_{2_{1}}r^{2}(\theta_{r}^{1}-\theta_{r}^{2}-\pi )\\
   &+&\sum_{j=2}^{n}\frac{a_{2_{j}}}{j-1}r^{j+1}(-\sin ((j-1)\theta_{r}^{2})+(-1)^{j-1}\sin ((j-1)\theta_{r}^{1}))\\
    &+&\sum_{j=2}^{n}\frac{b_{2_{j}}}{j-1}r^{j+1}(-\cos ((j-1)\theta_{r}^{2})+(-1)^{j-1}\cos ((j-1)\theta_{r}^{1}))
\end{eqnarray*}
\begin{eqnarray*}
     \Re \left( \int_{L_{h}^{3}}i\overline{h_{3}(z)}dz\right)&=&a_{3_{0}}r(\sin(\theta_{r}^{2})-\sin(\theta_{r}^{3}))+b_{3_{0}}r(-\cos(\theta_{r}^{2})+\cos(\theta_{r}^{3}))+a_{3_{1}}r^{2}(\theta_{r}^{2}-\theta_{r}^{3})\\
     &+&\sum_{j=2}^{n}\frac{a_{3_{j}}}{j-1}r^{j+1}(\sin((j-1)\theta_{r}^{2})-\sin((j-1)\theta_{r}^{3}))\\
      &+&\sum_{j=2}^{n}\frac{b_{3_{j}}}{j-1}r^{j+1}(\cos((j-1)\theta_{r}^{2})-\cos((j-1)\theta_{r}^{3}))
\end{eqnarray*}
\begin{eqnarray*}
       \Re \left( \int_{L_{h}^{4}}i\overline{h_{2}(z)}dz\right)&=&a_{2_{0}}r(\sin(\theta_{r}^{3})-\sin(\theta_{r}^{1}))-b_{2_{0}}r(\cos(\theta_{r}^{3})+\cos(\theta_{r}^{1}))+a_{2_{1}}r^{2}(\theta_{r}^{3}+\theta_{r}^{1}-\pi )\\
       &+&\sum_{j=2}^{n}\frac{a_{2_{j}}}{j-1}r^{j+1}(\sin((j-1)\theta_{r}^{3})-(-1)^{j}\sin((j-1)\theta_{r}^{1}))\\
        &+&\sum_{j=2}^{n}\frac{b_{2_{j}}}{j-1}r^{j+1}(\cos((j-1)\theta_{r}^{3})-(-1)^{j-1}\cos((j-1)\theta_{r}^{1}))\\
\end{eqnarray*}

It follows that
\begin{eqnarray*}
    M_{2}(r)&=&(2a_{1_{0}}-2a_{2_{0}})r\sin(\theta_{r}^{1})+(a_{3_{0}}-a_{2_{0}})r(\sin(\theta_{r}^{2})-\sin(\theta_{r}^{3}))\\
    &+&(b_{2_{0}}-b_{3_{0}})r(\cos(\theta_{r}^{2})-\cos(\theta_{r}^{3}))+(-2a_{1_{1}}+2a_{2_{1}})r^{2}\theta_{r}^{1}\\
    &+&a_{2_{1}}r^{2}(\theta_{r}^{3}-\theta_{r}^{2}-2\pi )+a_{3_{1}}r^{2}(\theta_{r}^{2}-\theta_{r}^{3})\\
    &+&\sum_{j=2}^{n}\frac{(-1)^{j-1}(-2a_{1_{j}}+2a_{2_{j}})}{j-1}r^{j+1}\sin((j-1)\theta_{r}^{1})\\
   &+& \sum_{j=2}^{n}\frac{(-a_{2_{j}}+a_{3_{j}})}{j-1}r^{j+1}(\sin((j-1)\theta_{r}^{2})-\sin((j-1)\theta_{r}^{3}))\\
    &+& \sum_{j=2}^{n}\frac{(-b_{2_{j}}+b_{3_{j}})}{j-1}r^{j+1}(\cos((j-1)\theta_{r}^{2})-\cos((j-1)\theta_{r}^{3})).
\end{eqnarray*}

Specifically, when $n=4$,
\begin{eqnarray*}
    M_{2}(r)=\sum_{j=1}^{15}A_{j}f_{j}(r),
\end{eqnarray*}
where
\[
\begin{aligned}
f_{1}(r) &= r\sin(\theta_{r}^{1}), &\quad f_{2}(r) &= r(\sin(\theta_{r}^{2}) - \sin(\theta_{r}^{3})), \\
f_{3}(r) &= r(\cos(\theta_{r}^{2}) - \cos(\theta_{r}^{3})), &\quad f_{4}(r) &= r^{2} \theta_{r}^{1}, \\
f_{5}(r) &= r^{2}( \theta_{r}^{3}- \theta_{r}^{2}-2\pi ), &\quad f_{6}(r) &= r^{2}(\theta_{r}^{2}- \theta_{r}^{3}), \\
f_{7}(r) &= r^{3} \sin(\theta_{r}^{1}), &\quad f_{8}(r) &= r^{3} (\sin(\theta_{r}^{2}) - \sin(\theta_{r}^{3})), \\
f_{9}(r) &= r^{3} (\cos(\theta_{r}^{2}) - \cos(\theta_{r}^{3})), &\quad f_{10}(r) &= r^{4} \sin(2\theta_{r}^{1}), \\
f_{11}(r) &= r^{4} (\sin(2\theta_{r}^{2}) - \sin(2\theta_{r}^{3})), &\quad f_{12}(r) &= r^{4} (\cos(2\theta_{r}^{2}) - \cos(2\theta_{r}^{3})), \\
f_{13}(r) &= r^{5} \sin(3\theta_{r}^{1}), &\quad f_{14}(r) &= r^{5} (\sin(3\theta_{r}^{2}) - \sin(3\theta_{r}^{3})), \\
f_{15}(r) &= r^{5} (\cos(3\theta_{r}^{2}) - \cos(3\theta_{r}^{3})).
\end{aligned}
\]
and
\[
\begin{aligned}
A_{1} &= (2a_{1_{0}} - 2a_{2_{0}}), &\quad A_{2} &= (a_{3_{0}} - a_{2_{0}}), \\
A_{3} &= (b_{2_{0}} - b_{3_{0}}), &\quad A_{4} &= (-2a_{1_{1}} + 2a_{2_{1}}), \\
A_{5} &= a_{2_{1}}, &\quad A_{6} &= a_{3_{1}}, \\
A_{7} &= (2a_{1_{2}} - 2a_{2_{2}}), &\quad A_{8} &= (-a_{2_{2}} + a_{3_{2}}), \\
A_{9} &= (-b_{2_{2}} + b_{3_{2}}), &\quad A_{10} &= \frac{-2a_{1_{3}} + 2a_{2_{3}}}{2}, \\
A_{11} &= \frac{-a_{2_{3}} + a_{3_{3}}}{2}, &\quad A_{12} &= \frac{-b_{2_{3}} + b_{3_{3}}}{2}, \\
A_{13} &= \frac{2a_{1_{4}} - 2a_{2_{4}}}{3}, &\quad A_{14} &= \frac{-a_{2_{4}} + a_{3_{4}}}{3}, \\
A_{15} &= \frac{-b_{2_{4}} + b_{3_{4}}}{3}.
\end{aligned}
\]

Using the expressions for $\theta_{r}^{2}$ and $\theta_{r}^{3}$ together with the identities
\begin{eqnarray*}
    \cos(x \pm y) = \cos x \cos y \mp \sin x \sin y \quad \hbox{and } \quad \sin(x \pm y) = \sin x \cos y \pm \cos x \sin y,
\end{eqnarray*}
we obtain that $f_{2}(r)=-\frac{4\sqrt{5}}{5}r\sin \theta_{r}^{4}$, $f_{3}(r)=\frac{2\sqrt{5}}{5}r\sin \theta_{r}^{4}$, $f_{5}(r)=2r^{2}(\theta_{r}^{4}-\pi )$, $f_{6}(r)=-2r^{2}\theta_{r}^{4}$, $f_{8}(r)=-\frac{4\sqrt{5}}{5}r^{3}\sin \theta_{r}^{4}$, $f_{9}(r)=\frac{2\sqrt{5}}{5}r^{3}\sin \theta_{r}^{4}$, $f_{11}(r)=-\frac{6}{5}r^{4}\sin 2\theta_{r}^{4}$, $f_{12}(r)=\frac{8}{5}r^{4}\sin 2\theta_{r}^{4}$, $f_{14}(r)=-\frac{4\sqrt{5}}{25}r^{5}\sin 3\theta_{r}^{4}$ and $f_{15}(r)=\frac{22\sqrt{5}}{25}r^{5}\sin 3\theta_{r}^{4}$, where $\theta_{r}^{4}=\arccos \left(\frac{2+\frac{r^{2}}{2}}{\sqrt{5}r} \right)$. 
Hence, it follows that $f_{2}(r)=-2f_{3}(r)$,  $f_{8}(r)=-2f_{9}(r)$, $f_{11}(r)=-\frac{3}{4}f_{12}(r)$, and $f_{14}(r)=-\frac{2}{11}f_{15}(r)$. 
Therefore, we can rewrite $M_{2}(r)$ as
\begin{eqnarray*}
    M_{2}(r)=\sum_{j=1}^{11}B_{j}g_{j}(r),
\end{eqnarray*}
where 
\[
\begin{aligned}
g_{1}(r) &= r\sin(\theta_{r}^{1}), &\quad g_{2}(r) &= r\sin(\theta_{r}^{4}),&\quad g_{3}(r) &=r^{2}\theta_{r}^{1}, \\
g_{4}(r) &=r^{2}(\theta_{r}^{4}-\pi ),&\quad g_{5}(r) &=r^{2}, &\quad g_{6}(r) &=r^{3}\sin (\theta_{r}^{1}),\\
g_{7}(r) &=r^{3}\sin (\theta_{r}^{4}), &\quad g_{8}(r) &=r^{4}\sin (2\theta_{r}^{1}), & \quad  g_{9}(r) &=r^{4}\sin (2\theta_{r}^{4})\\
g_{10}(r) &=r^{5}\sin (3\theta_{r}^{1}), &\quad g_{11}(r) &=r^{5}\sin (3\theta_{r}^{4}) .
\end{aligned}
\]
and
\[
\begin{aligned}
B_{1} &= A_{1}, &\quad 
B_{2} &= -\frac{4\sqrt{5}}{5}A_{2}+\frac{2\sqrt{5}}{5}A_{3}, &\quad 
B_{3} &= A_{4}, \\
B_{4} &= 2A_{5}-2A_{6}, &\quad 
B_{5} &= -2\pi A_{6}, &\quad 
B_{6} &= A_{7}, \\
B_{7} &= -\frac{4\sqrt{5}}{5}A_{8}+\frac{2\sqrt{5}}{5}A_{9}, &\quad
B_{8} &= A_{10}, &\quad 
B_{9} &= -\frac{6}{5}A_{11}+\frac{8}{5}A_{12}, \\
B_{10} &= A_{13}, &\quad 
B_{11} &= -\frac{4\sqrt{5}}{25}A_{14}+\frac{22\sqrt{5}}{25}A_{15}. &
\end{aligned}
\]

Note that $W_{11}(2)\approx 2.237065\times 10^{19}\neq 0$. Therefore, the functions $g_{1},\cdots , g_{11}$ are linearly independent. In addition, $g_{5}$ maintains a constant sign over the interval $\left( \sqrt{2(3-\sqrt{5})},3\right)$, and the coefficients $B_{1},\cdots , B_{11}$ can be chosen arbitrarily. Consequently, by Lemma \ref{lem1}(b), it follows that one can select these coefficients so that $M_{2}(r)$ has at least $10$ simple zeros. Therefore, system (\ref{eq.it3}) has at least $10$ limit cycles bifurcating from the center $\dot{z}=iz$, which intersect both $|z+2|=1$ and $|z-2-i|=1$.
Finally, we consider $M_{3}$. We have that
\begin{eqnarray*}
    \Re \left( \int_{m_{h}^{3}}i\overline{h_{3}(z)}dz\right)&=&a_{3_{0}}r(\sin(\theta_{r}^{2})-\sin(\theta_{r}^{3}))-b_{3_{0}}r(\cos(\theta_{r}^{2})-\cos(\theta_{r}^{3}))+a_{3_{1}}r^{2}(\theta_{r}^{2}-\theta_{r}^{3})\\
    &+&\sum_{j=2}^{n}\frac{a_{3_{j}}}{j-1}r^{j+1}(\sin((j-1)\theta_{r}^{2})-\sin((j-1)\theta_{r}^{3}))\\
    &+&\sum_{j=2}^{n}\frac{b_{3_{j}}}{j-1}r^{j+1}(\cos((j-1)\theta_{r}^{2})-\cos((j-1)\theta_{r}^{3}))
\end{eqnarray*}
and
\begin{eqnarray*}
    \Re \left( \int_{m_{h}^{2}}i\overline{h_{2}(z)}dz\right)&=&-a_{2_{0}}r(\sin(\theta_{r}^{2})-\sin(\theta_{r}^{3}))+b_{2_{0}}r(\cos(\theta_{r}^{2})-\cos(\theta_{r}^{3}))-2\pi a_{2_{1}}r^{2}\\
    &-&a_{2_{1}}r^{2}(\theta_{r}^{2}-\theta_{r}^{3})-\sum_{j=2}^{n}\frac{a_{2_{j}}}{j-1}r^{j+1}(\sin((j-1)\theta_{r}^{2})-\sin((j-1)\theta_{r}^{3}))\\
    &-&\sum_{j=2}^{n}\frac{b_{2_{j}}}{j-1}r^{j+1}(\cos((j-1)\theta_{r}^{2})-\cos((j-1)\theta_{r}^{3}))
\end{eqnarray*}
As a result 
\begin{eqnarray*}
    M_{3}(r)&=&(a_{3_{0}}-a_{2_{0}})r(\sin(\theta_{r}^{2})-\sin(\theta_{r}^{3}))+(b_{2_{0}}-b_{3_{0}})r(\cos(\theta_{r}^{2})-\cos(\theta_{r}^{3}))-2\pi a_{2_{1}}r^{2}\\
    &+&(a_{3_{1}}-a_{2_{1}})r^{2}(\theta_{r}^{2}-\theta_{r}^{3})+\sum_{j=2}^{n}\frac{(a_{3_{j}}-a_{2_{j}})}{j-1}r^{j+1}(\sin((j-1)\theta_{r}^{2})-\sin((j-1)\theta_{r}^{3}))\\
    &+&\sum_{j=2}^{n}\frac{(b_{3_{j}}-b_{2_{j}})}{j-1}r^{j+1}(\cos((j-1)\theta_{r}^{2})-\cos((j-1)\theta_{r}^{3})).
\end{eqnarray*}

In particular, for $n=4$,
\begin{eqnarray*}
    M_{3}(r)=\sum_{j=1}^{10}A_{j}f_{j}(r),
\end{eqnarray*}
where
\[
\begin{aligned}
f_{1}(r) &= r(\sin(\theta_{r}^{2})-\sin(\theta_{r}^{3})), &\quad
f_{2}(r) &= r(\cos(\theta_{r}^{2})-\cos(\theta_{r}^{3})), \\
f_{3}(r) &= r^{2}, &\quad
f_{4}(r) &= r^{2}(\theta_{r}^{2}-\theta_{r}^{3}), \\
f_{5}(r) &= r^{3}(\sin(\theta_{r}^{2})-\sin(\theta_{r}^{3})), &\quad
f_{6}(r) &= r^{3}(\cos(\theta_{r}^{2})-\cos(\theta_{r}^{3})), \\
f_{7}(r) &= r^{4}(\sin(2\theta_{r}^{2})-\sin(2\theta_{r}^{3})), &\quad
f_{8}(r) &= r^{4}(\cos(2\theta_{r}^{2})-\cos(2\theta_{r}^{3})), \\
f_{9}(r) &= r^{5}(\sin(3\theta_{r}^{2})-\sin(3\theta_{r}^{3})), &\quad
f_{10}(r) &= r^{5}(\cos(3\theta_{r}^{2})-\cos(3\theta_{r}^{3})),
\end{aligned}
\]
and 
\[
\begin{aligned}
A_1 &= a_{3_{0}}-a_{2_{0}}, &\quad
A_2 &= b_{2_{0}}-b_{3_{0}}, \\
A_3 &= -2\pi a_{2_{1}}, &\quad
A_4 &= a_{3_{1}}-a_{2_{1}}, \\
A_5 &= a_{3_{2}}-a_{2_{2}}, &\quad
A_6 &= b_{3_{2}}-b_{2_{2}}, \\
A_7 &= \frac{a_{3_{3}}-a_{2_{3}}}{2}, &\quad
A_8 &= \frac{b_{3_{3}}-b_{2_{3}}}{2}, \\
A_9 &= \frac{a_{3_{4}}-a_{2_{4}}}{3}, &\quad
A_{10} &= \frac{b_{3_{4}}-b_{2_{4}}}{3}.
\end{aligned}
\]

Proceeding analogously to the case of $M_{2}(r)$, and using the same trigonometric identities, we obtain that
\begin{eqnarray*}
    M_{3}(r) &=&B_{1}r \sin(\theta_r^4) + B_{2} r^2 \theta_r^4 +B_{3} r^2  
+ B_{4} r^3 \sin(\theta_r^4) + B_{5} r^4 \sin(2 \theta_r^4) + B_{6} r^5 \sin(3 \theta_r^4),
\end{eqnarray*}
where
\[
\begin{aligned}
B_1 &= \frac{4\sqrt{5}}{5}(a_{2_{0}}-a_{3_{0}})+\frac{2\sqrt{5}}{5}(b_{2_{0}}-b_{3_{0}}), &\quad
B_2 &=2(a_{2_{1}}-a_{3_{1}}), \\
B_3 &= -2\pi a_{2_{1}}, &\quad
B_4 &= \frac{4\sqrt{5}}{5}(a_{2_{2}}-a_{3_{2}})+\frac{2\sqrt{5}}{5}(b_{3_{2}}-b_{2_{2}}), \\
B_5 &= \frac{3}{5}(a_{2_{3}}-a_{3_{3}})+\frac{4}{5}(b_{3_{3}}-b_{2_{3}}), &\quad
B_6 &= \frac{4\sqrt{5}}{75}(a_{2_{4}}-a_{3_{4}})+\frac{22\sqrt{5}}{75}(b_{3_{4}}-b_{2_{4}}), \\
\end{aligned}
\]

Let
\begin{eqnarray*}
    \mathcal{F}=\{ r\sin (\theta_{r}^{4}),r^3\sin(\theta_{r}^{4}),r^{4}\sin(2\theta_{r}^{4}),r^{5}\sin(3\theta_{r}^{4}),r^{2},r^{2}\theta_{r}^{4}\}
\end{eqnarray*}
be an ordered set of functions. Note that $\theta_r^4$ fails to be differentiable at $r=\sqrt{2(3+\sqrt5)}$, where the level curve becomes tangent to $|z-2-i|=1$; so $\mathcal{F}$ is analytic only on the open interval $\left(3,\sqrt{2(3+\sqrt5)}\right)$, and Lemma \ref{lema2} must be applied on a closed subinterval of it. Fix $[a,b]=[3.05,3.2]\subset\left(3,\sqrt{2(3+\sqrt5)}\right)$. By Appendix \ref{app.w} (Case 3), all six Wronskians $W_{0},\ldots,W_{5}$ are nonzero throughout $\left(3,\sqrt{2(3+\sqrt{5})}\right)$, in particular on $[3.05,3.2]$, so $\mathcal{F}$ is an ECT-system there.
Thus, by Lemma \ref{lema2}, applied on $[3.05,3.2]$, any element of $\operatorname{Span}(\mathcal F)$ has at most $5$ simple zeros on $[3.05,3.2]$, and every configuration of $m\leq 5$ prescribed simple zeros in $[3.05,3.2]$ is realized by some element of $\operatorname{Span}(\mathcal F)$. Since the coefficients $B_1,\ldots,B_6$ are independent linear functions of the free parameters $a_{2_j},a_{3_j},b_{2_j},b_{3_j}$, such a configuration, with its five zeros placed in $[3.05,3.2]\subset\left(3,\sqrt{2(3+\sqrt5)}\right)$, is realizable by the perturbation. Therefore, system (\ref{eq.it3}) has at least $5$ limit cycles bifurcating from the center $\dot{z}=iz$, which intersect only $|z-2-i|=1$.
\end{proof}

\renewcommand{\proofname}{Proof of statement (b) of Theorem \ref{teoc}}
\begin{proof}
    The following system
 \begin{eqnarray}\label{eq.ex.teocc}
\begin{cases}
\dot{z}=iz+\varepsilon \left( h_{1}(z) \right) \quad \mbox{if} \quad |z+2|<1,\\
\dot{z}=iz+\varepsilon \left( h_{2}(z) \right) \quad \mbox{if} \quad |z+2|>1 \quad \hbox{and }\quad |z-2-i|>1  \\
\dot{z}=iz+\varepsilon \left( h_{3}(z)\right) \quad \mbox{if} \quad |z-2-i|<1,
\end{cases}
\end{eqnarray}
has at least $14$ limit cycles for sufficiently small $\varepsilon $, where
\begin{eqnarray*}
    h_{1}(z)&=& 0.347564393426+0.103145110396z-0.408654841162z^{2}\\
    &-&0.330744674151z^{3}-0.063628043173z^{4}\\
    h_{2}(z)&=&0.005927528247z\\
    h_{3}(z)&=&0.291594843305-0.070287496338z-0.559016994375z^{2}\\
    &+&0.701345221921z^{3}- 0.505728132883z^{4}
\end{eqnarray*}
Indeed, note that the first-order Melnikov function associated with (\ref{eq.ex.teocc}) and corresponding to the first family of periodic orbits is given by
\begin{align*}
M_{1}(r) =\;& 
     0.695128786853\, r\sin(\theta_{r}^{1}) 
    -0.037243758387\,r^{2} 
    -0.194435164299 r^{2} \theta_{r}^{1} \\
    & -0.817309682324\,r^{3} \sin(\theta_{r}^{1}) 
    +0.330744674151\,r^{4} \sin(2\theta_{r}^{1}) 
     -0.042418695449\,r^{5} \sin(3\theta_{r}^{1})
\end{align*}
and $M_{1}(r)$ has two zeros in $\left(1,\sqrt{2(3-\sqrt{5})}\right)$
\begin{eqnarray*}
    r_{1}\approx 1.042794, \quad r_{2}\approx 1.219573.
\end{eqnarray*}
Differentiating $M_{1}$, we find that $M_{1}'(r)$ vanishes on $\left(1,\sqrt{2(3-\sqrt{5})}\right)$ exactly at $r\approx 1.115244$, distinct from $r_{1},r_{2}$ above. Hence, $M_{1}'(r_{j})\neq 0$ for $j=1,2$ so both zeros of $M_{1}(r)$ are simple.
The first-order Melnikov function corresponding to the third family is given
by
\begin{align*}
M_{3}(r) =\;& 
    -0.521620713215\, r\sin(\theta_{r}^{4}) 
    - 0.037243758387\,r^{2} 
    +0.152430049170 r^{2} \theta_{r}^{4} \\
    & +\,r^{3} \sin(\theta_{r}^{4}) 
    - 0.420807133153\,r^{4} \sin(2\theta_{r}^{4}) 
    +0.060311599107\,r^{5} \sin(3\theta_{r}^{4})
\end{align*}
and $M_{3}(r)$ has two zeros in $\left(3,\sqrt{2(3+\sqrt{5})}\right)$
\begin{eqnarray*}
    r_{1}\approx  3.022871, \quad r_{2}\approx 3.215850.
\end{eqnarray*}
Differentiating $M_{3}(r)$, we find that $M_{3}'(r)$ vanishes on $\left(3,\sqrt{2(3+\sqrt{5})}\right)$ at $r\approx 3.146898$, distinct from $r_{1},r_{2}$ above. Hence, $M_{3}'(r_{j})\neq0 $ for $j=1,2$ so both zeros of $M_{3}(r)$ are simple.
\begin{figure}[!htb]
\centering
\begin{minipage}{0.48\textwidth}
    \centering
    \includegraphics[width=\textwidth]{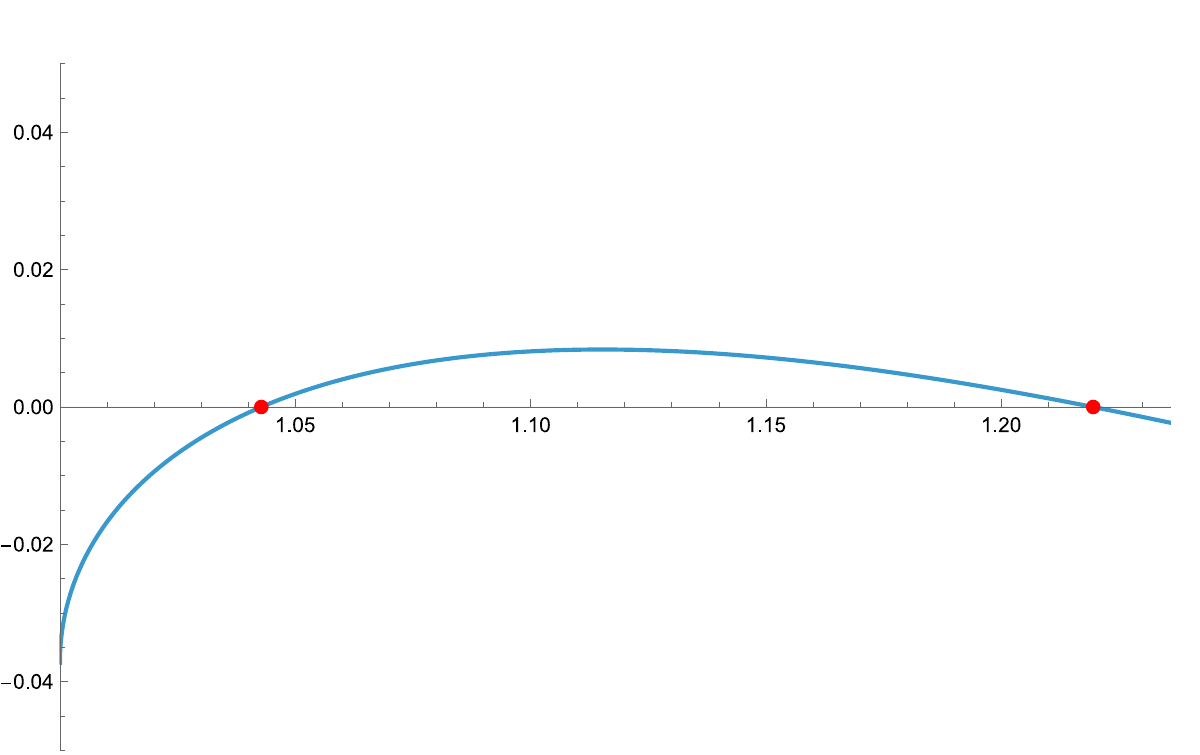}
\end{minipage}
\hspace{-0.3cm}
\begin{minipage}{0.48\textwidth}
    \centering
    \includegraphics[width=\textwidth]{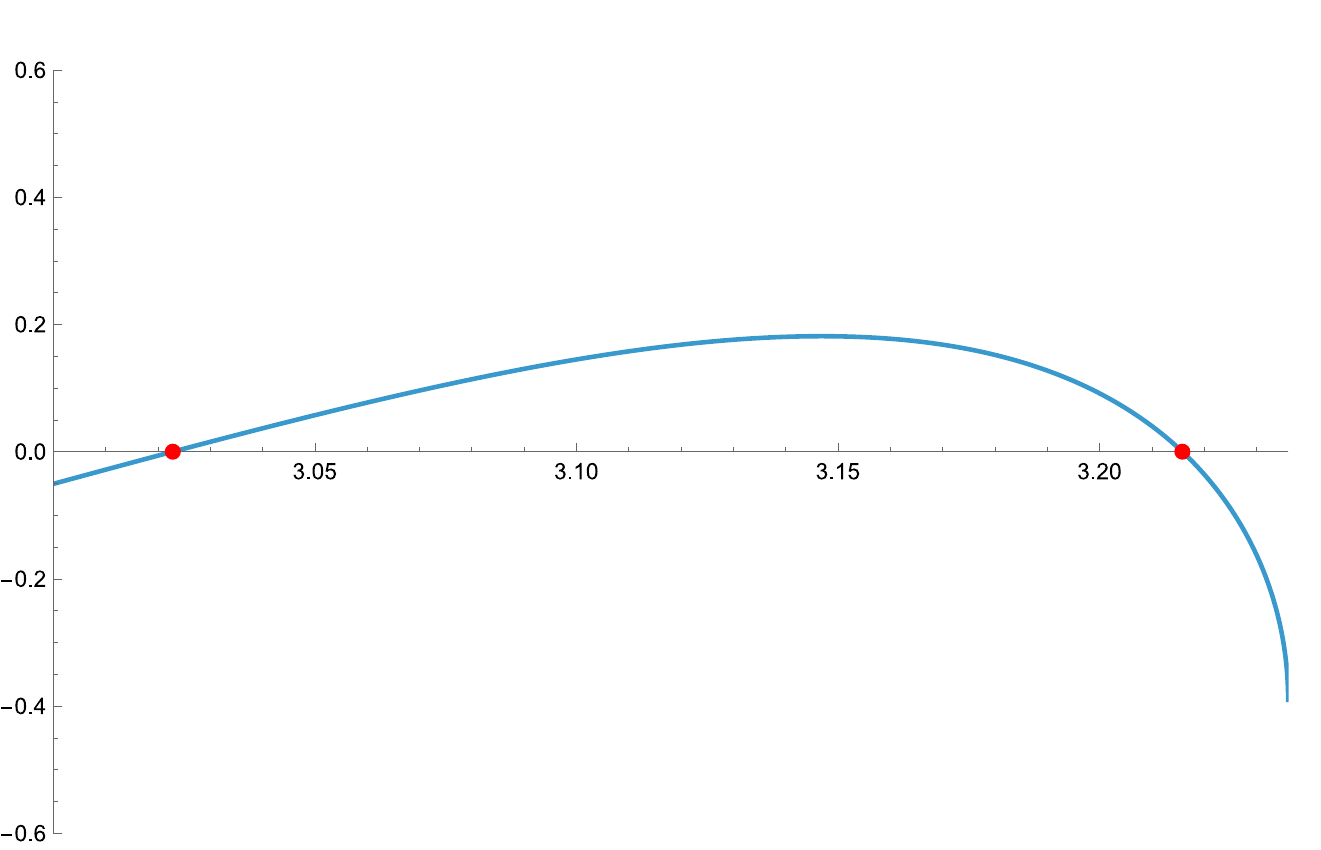}
\end{minipage}
\caption{Graph of the Melnikov functions $M_1(r)$ and $M_{3}(r)$ associated with system (\ref{eq.ex.teocc}).}
\end{figure}
Finally, the  first-order Melnikov function of the second family is given
by
\begin{align*}
M(r) =\;& 
     0.695128786853\, r\sin(\theta_{r}^{1}) 
     -0.521620713215\,r\sin(\theta_{r}^{4}) \\
   & -0.194435164299 r^{2} \theta_{r}^{1}+0.152430049170r^{2}(\theta_{r}^{4}-\pi ) + 0.441629364272\, r^{2}\\
    & -0.817309682324\,r^{3} \sin(\theta_{r}^{1}) +\,r^{3} \sin(\theta_{r}^{4})\\ 
    &+0.330744674151\,r^{4} \sin(2\theta_{r}^{1}) - 0.420807133153\,r^{4} \sin(2\theta_{r}^{4})\\
    &-0.042418695449r^{5} \sin(3\theta_{r}^{1})+0.060311599107r^{5} \sin(3\theta_{r}^{4}),
\end{align*}
and admits ten zeros
\begin{eqnarray*}
    r_{1}& \approx &  1.317699, \quad r_{2}\approx 1.376726, \quad r_{3}\approx 1.580858, \quad r_{4} \approx 1.767719, \quad r_{5}\approx  1.809213
    \\ r_{6}&\approx &1.884499, \quad r_{7}\approx 2.459783, \quad r_{8}\approx 2.695006, \quad r_{9} \approx 2.755498, \quad r_{10}\approx 2.808856
\end{eqnarray*}
Moreover, differentiating $M(r)$, one finds that $M'(r)$ vanishes on $\left( \sqrt{2(3-\sqrt{5})},3\right)$ exactly at the points
\begin{eqnarray*}
    r &\approx& 1.337763,\quad 1.439581,\quad 1.639157,\quad 1.787038,\quad \quad 1.855654,\\
\quad & & 2.270152,\quad 2.576217,\quad 2.726258,\quad 2.790277
\end{eqnarray*}
none of which coincides with $r_{1},\cdots , r_{10}$ above. Hence, $M'(r_{j})\neq 0$ for $j=1,\cdots 10$ so all ten zeros of $M(r)$ are simple.
Therefore, for $\varepsilon $ sufficiently small the system  (\ref{eq.ex.teocc}) has at least fourteen  limit cycles.
\begin{figure}[!htb]
\centering
\includegraphics[scale=0.68]{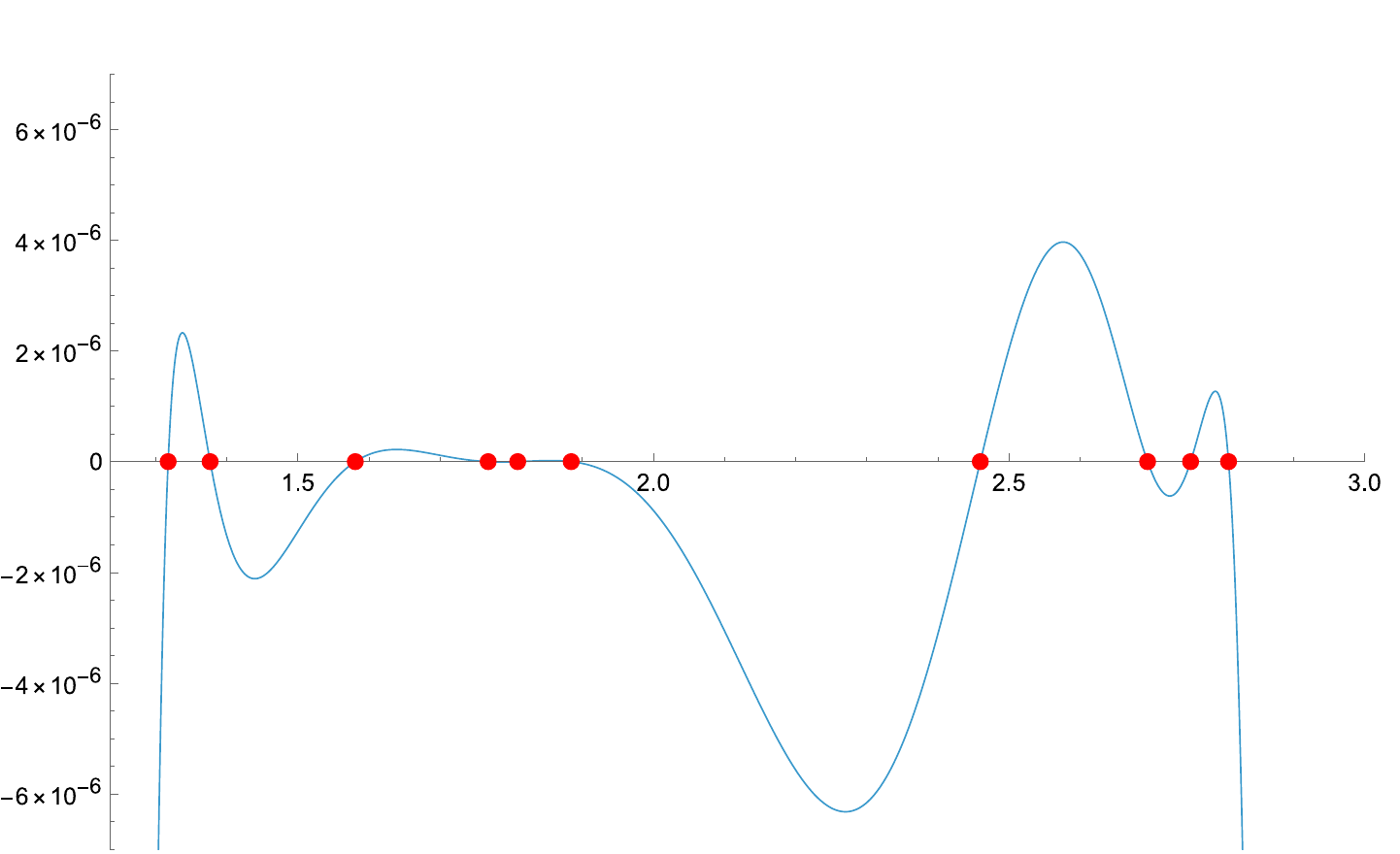}
\caption{Graph of the Melnikov function $M(r)$ associated with system (\ref{eq.ex.teocc}). }
\label{imgteoc}
\end{figure}
\end{proof}
\section{Conclusion}
In this work, we studied the bifurcation of limit cycles from the periodic orbits of the linear holomorphic center $\dot{z}=iz$ under piecewise polynomial holomorphic perturbations with three zones separated by two external, disjoint circles. Three geometric configurations were considered: equal, aligned circles (Theorem \ref{teoa}); circles of different radii, still aligned (Theorem \ref{teob}); and equal circles that are not aligned (Theorem \ref{teoc}). In every configuration beyond the most symmetric one, the loss of symmetry (different radii or non-aligned centers) splits the periodic orbits of the unperturbed center into three disjoint annuli, one crossing both circles and two crossing only one of them, and this decomposition is what allows limit cycles to be produced simultaneously in all three annuli by a single common choice of the perturbation parameters. Although the Melnikov function was obtained for perturbations of arbitrary degree $n$, the complexity of the resulting expression led us to present the results only for the case $n=4$, which yields the bounds and explicit examples of Theorems \ref{teob}(b) and \ref{teoc}(b), with $16$ and $14$ limit cycles, respectively.

\appendix
\section{Proof of Theorem \ref{teomelnikov}}\label{sec6}

In this appendix, we follow the steps of \cite{refb8} and \cite{refb17}, adapting them to our specific case to derive the expression given in (\ref{funcaomelnikov}).
Consider the system (\ref{melnikov1}) and the solution that starts at the point $A_{1}(h)$. Let $A_{2,\varepsilon }(h)$ be the point of intersection of this solution with $\Sigma_{2} $. Denote by $\tilde{A}_{2,\varepsilon }(h)$ the point of intersection of the solution of system  (\ref{melnikov1})  starting at $A_{2, \varepsilon }(h)$ with $\tilde{\Sigma }_{2}$. Next, let $\tilde{A}_{1,\varepsilon }(h)$ be the point of intersection of the orbit starting at $\tilde{A}_{2,\varepsilon }(h)$ with $\tilde{\Sigma}_{1}$. Finally, denote by $A_{1, \varepsilon }(h)$ the point of intersection of the orbit starting at $\tilde{A}_{1,\varepsilon }(h)$ with $\Sigma_{1}$.
Using the Poincaré map of the system (\ref{melnikov1}), we define the displacement function
\begin{eqnarray}\label{eq.a1}
    H_{1}(A_{1,\varepsilon}(h))-H_{1}(A_{1}(h))=\varepsilon M(h)+O(\varepsilon^{2}),
\end{eqnarray}
where $M(h)$ is called the first-order Melnikov function of system  (\ref{melnikov1}).
\begin{figure}[!htb]
\centering
\includegraphics[scale=1.0]{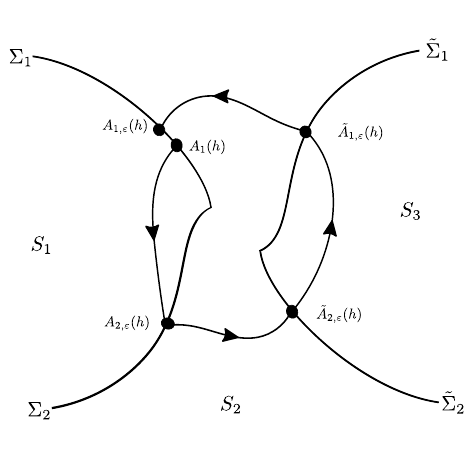}
\caption{The Poincaré  map related to perturbed system (\ref{melnikov1}). }
\label{img20}
\end{figure}
We have that
\begin{eqnarray*}
    H_{1}(A_{1,\varepsilon }(h))-H_{1}(A_{1}(h))&=&\left[ H_{1}(A_{1,\varepsilon }(h))-H_{2}(A_{1,\varepsilon }(h))\right]+\left[ H_{2}(A_{1,\varepsilon }(h))-H_{2}(\tilde{A}_{1,\varepsilon }(h))\right]\\
    &+&\left[ H_{2}(\tilde{A}_{1,\varepsilon }(h))-H_{3}(\tilde{A}_{1,\varepsilon }(h))\right] +\left[ H_{3}(\tilde{A}_{1,\varepsilon }(h))-H_{3}(\tilde{A}_{2,\varepsilon }(h))\right]\\
    &+&\left[ H_{3}(\tilde{A}_{2,\varepsilon }(h))-H_{2}(\tilde{A}_{2,\varepsilon }(h))\right]+\left[ H_{2}(\tilde{A}_{2,\varepsilon }(h))-H_{2}(A_{2,\varepsilon }(h))\right]\\
    &+&\left[ H_{2}(A_{2,\varepsilon }(h))-H_{1}(A_{2,\varepsilon }(h))\right]+\left[ H_{1}(A_{2,\varepsilon }(h))-H_{1}(A_{1}(h))\right]\\
    &=&l_{1}+l_{2}+l_{3}+l_{4}+l_{5}+l_{6}+l_{7}+l_{8}.
\end{eqnarray*}

Observe that
\begin{eqnarray*}
    l_{2}&=&\left[ H_{2}(A_{1,\varepsilon }(h))-H_{2}(\tilde{A}_{1,\varepsilon }(h))\right]\\
    &=&\int_{\widehat{\tilde{A}_{1,\varepsilon },A_{1,\varepsilon}}}dH_{2}(x,y)=\int_{\widehat{\tilde{A}_{1,\varepsilon },A_{1,\varepsilon}}} H_{2x}dx+H_{2y}dy\\
  &=&  \int_{\widehat{\tilde{A}_{1,\varepsilon },A_{1,\varepsilon}}}H_{2x}(H_{2y}+\varepsilon f_{2})dt+H_{2y}(-H_{2x}+\varepsilon g_{2})dt\\
  &=&\varepsilon \int_{\widehat{\tilde{A}_{1,\varepsilon },A_{1,\varepsilon}}}(H_{2x}f_{2}+H_{2y}g_{2})dt\\
  &=&\varepsilon \int_{\widehat{\tilde{A}_{1,\varepsilon },A_{1,\varepsilon}}}f_{2}(-dy+\varepsilon g_{2}dt)+g_{2}(dx-\varepsilon f_{2}dt)\\
  &=&\varepsilon \int_{\widehat{\tilde{A}_{1},A_{1}}} (g_{2}dx-f_{2}dy)+O(\varepsilon^{2} ).
\end{eqnarray*}

Therefore
\begin{eqnarray}\label{a.eq1}
    \frac{d}{d \varepsilon}l_{2}\Big|_{\varepsilon=0}=\int_{L_{h}^{4}} g_{2}dx-f_{2}dy.
\end{eqnarray}
Similarly, we can see that
\begin{align}
    \frac{d}{d \varepsilon}l_{4}\Big|_{\varepsilon=0} &= \int_{L_{h}^{3}} g_{3}dx - f_{3}dy, \label{a.eq2} \\
    \frac{d}{d \varepsilon}l_{6}\Big|_{\varepsilon=0} &= \int_{L_{h}^{2}} g_{2}dx - f_{2}dy, \label{a.eq3} \\
    \frac{d}{d \varepsilon}l_{8}\Big|_{\varepsilon=0} &= \int_{L_{h}^{1}} g_{1}dx - f_{1}dy. \label{a.eq4}
\end{align}

By the definition of $l_{8}$, we have that
\begin{eqnarray*}
    \frac{d}{d\varepsilon  }l_{8}= \langle \nabla H_{1}(A_{2,\varepsilon}),\left. \frac{d}{d\varepsilon }A_{2,\varepsilon }\right|_{\varepsilon =0} \rangle =\langle  \nabla H_{1} (A_{2}),\left. \frac{d}{d\varepsilon }A_{2, \varepsilon }\right|_{\varepsilon =0}\rangle .
\end{eqnarray*}

Since $\varphi_{2}(A_{2,\varepsilon })=0$, then
\begin{eqnarray*}
    \langle \nabla \varphi_{2}(A_{2}),\frac{d}{d\varepsilon }A_{2,\varepsilon } \rangle =0.
\end{eqnarray*}

Therefore,  there exists $\lambda_{2}\in \mathbb{R}$ such that  
\begin{eqnarray*}
    \left. \frac{d}{d\varepsilon }A_{2, \varepsilon }\right|_{\varepsilon =0}=\lambda_{2}\nabla\varphi_{2}^{\perp}(A_{2}),
\end{eqnarray*}
and consequently
\begin{eqnarray}\label{a.eq5}
    \frac{d}{d\varepsilon}l_{8}=\lambda_{2}\langle \nabla H_{1}(A_{2}),\nabla \varphi_{2}^{\perp}(A_{2})\rangle .
\end{eqnarray}

From equations (\ref{a.eq4}) and (\ref{a.eq5}) it follows that
\begin{eqnarray}\label{a.eq6}
    \lambda_{2}=\dfrac{1}{\langle\nabla H_{1} (A_{2}),\nabla \varphi_{2}^{\perp }(A_{2}) \rangle }\int_{L_{h}^{1}} g_{1}dx-f_{1}dy.
\end{eqnarray}

Similarly
\begin{equation} \label{a.eq7}
\begin{aligned}
  \left.  \frac{d}{d\varepsilon }l_{2}\right|_{\varepsilon =0} &= \langle \nabla H_{2}(A_{1}),\left.\frac{d}{d\varepsilon } A_{1,\varepsilon }\right|_{\varepsilon=0} \rangle -\langle \nabla H_{2}(\tilde{A}_{1}),\left.\frac{d}{d\varepsilon } \tilde{A}_{1,\varepsilon }\right|_{\varepsilon=0}\rangle \\
  &= \lambda_{1}\langle \nabla H_{2}(A_{1}),\nabla \varphi _{1}^{\perp}(A_{1}) \rangle-\tilde{\lambda }_{1}\langle \nabla H_{2}(\tilde{A}_{1}),\nabla \tilde{ \varphi }_{1}^{\perp}(\tilde{A}_{1}) \rangle ,
\end{aligned}
\end{equation}
\begin{equation} \label{a.eq8}
\begin{aligned}
  \left.  \frac{d}{d\varepsilon }l_{4}\right|_{\varepsilon =0} &= \langle \nabla H_{3}(\tilde{A}_{1}),\left.\frac{d}{d\varepsilon } \tilde{A}_{1,\varepsilon }\right|_{\varepsilon=0} \rangle -\langle \nabla H_{3}(\tilde{A}_{2}),\left.\frac{d}{d\varepsilon } \tilde{A}_{2,\varepsilon }\right|_{\varepsilon=0}\rangle \\
  &= \tilde{\lambda_{1}}\langle \nabla H_{3}(\tilde{A}_{1}),\nabla \tilde{ \varphi }_{1}^{\perp}(\tilde{A}_{1}) \rangle-\tilde{\lambda_{2}} \langle \nabla H_{3}(\tilde{A}_{2}),\nabla  \tilde{\varphi}_{2}^{\perp}(\tilde{A}_{2}) \rangle ,
\end{aligned}
\end{equation}
\begin{equation} \label{a.eq9}
\begin{aligned}
  \left.  \frac{d}{d\varepsilon }l_{6}\right|_{\varepsilon =0} &= \langle \nabla H_{2}(\tilde{A}_{2}),\left.\frac{d}{d\varepsilon } \tilde{A}_{2,\varepsilon }\right|_{\varepsilon=0} \rangle -\langle \nabla H_{2}(A_{2}),\left.\frac{d}{d\varepsilon } A_{2,\varepsilon }\right|_{\varepsilon=0}\rangle \\
  &= \tilde{\lambda_{2}}\langle \nabla H_{2}(\tilde{A}_{2}),\nabla \tilde{ \varphi }_{2}^{\perp}(\tilde{A}_{2}) \rangle-\lambda_{2}\langle \nabla H_{2}(A_{2}),\nabla  \varphi _{2}^{\perp}(A_{2}) \rangle .
\end{aligned}
\end{equation}

From equations (\ref{a.eq3}), (\ref{a.eq6}) and (\ref{a.eq9}), it follows that
\begin{equation}
\begin{aligned}\label{a.eq10}
    \tilde{\lambda}_{2}=&\frac{1}{\langle\nabla H_{2}(\tilde{A}_{2}),\nabla \tilde{\varphi}_{2}^{\perp}(\tilde{A}_{2}) \rangle}\int_{L_{h}^{2}} g_{2}dx-f_{2}dy\\
      &+\frac{\langle \nabla H_{2}(A_{2}),\nabla \varphi_{2}^{\perp}(A_{2}) \rangle}{\langle\nabla H_{2}(\tilde{A}_{2}),\nabla \tilde{\varphi}_{2}^{\perp}(\tilde{A}_{2}) \rangle \langle\nabla H_{1}(A_{2}),\nabla \varphi_{2}^{\perp}(A_{2}) \rangle}\int_{L_{h}^{1}} g_{1}dx-f_{1}dy.
\end{aligned}
\end{equation}

From equations (\ref{a.eq2}), (\ref{a.eq8}) and (\ref{a.eq10}), it follows that 
\begin{equation}
\begin{aligned}\label{a.eq11}
\tilde{\lambda}_{1}=\frac{1}{\langle \nabla H_{3}(\tilde{A}_{1}),\nabla \tilde{\varphi}_{1}^{\perp}(\tilde{A}_{1})\rangle}\int_{L_{h}^{3}} g_{3}dx-f_{3}dy+\tilde{\lambda}_{2}\frac{\langle \nabla H_{3}(\tilde{A}_{2}),\nabla\tilde{\varphi}_{2}^{\perp}(\tilde{A}_{2})\rangle}{\langle \nabla H_{3}(\tilde{A}_{1}),\nabla \tilde{\varphi}_{1}^{\perp}(\tilde{A}_{1})\rangle}.
\end{aligned}
\end{equation}

From equations (\ref{a.eq1}), (\ref{a.eq7}) and (\ref{a.eq11}), it follows that 
\begin{equation}
\begin{aligned}\label{a.eq12}
\lambda_{1}=\frac{1}{\langle \nabla H_{2}(A_{1}),\nabla \varphi_{1}^{\perp }(A_{1}) \rangle}\int_{L_{h}^{4}} g_{2}dx-f_{2}dy+\tilde{\lambda}_{1}\frac{\langle \nabla H_{2}(\tilde{A}_{1}),\nabla \tilde{\varphi }_{1}^{\perp }(\tilde{A}_{1})\rangle}{\langle \nabla H_{2}(A_{1}),\nabla \varphi_{1}^{\perp }(A_{1})\rangle}.
\end{aligned}
\end{equation}

It remains to compute the derivatives of the terms $l_{1}, l_{3}, l_{5}, l_{7}$. Since $l_1=H_1(A_{1,\varepsilon})-H_2(A_{1,\varepsilon})$ and $\left.\frac{d}{d\varepsilon}A_{1,\varepsilon}\right|_{\varepsilon=0}=\lambda_1\nabla\varphi_1^\perp(A_1)$, we have
\begin{equation}
\begin{aligned}\label{a.eq13}
\left.\frac{d}{d\varepsilon}l_{1}\right|_{\varepsilon=0}
&=\big\langle \nabla H_{1}(A_{1})-\nabla H_{2}(A_{1}),\ \lambda_1\nabla\varphi_{1}^{\perp}(A_{1})\big\rangle\\
&=\lambda_{1}\Big(\langle \nabla H_{1}(A_{1}),\nabla \varphi_{1}^{\perp}(A_{1})\rangle-\langle \nabla H_{2}(A_{1}),\nabla \varphi_{1}^{\perp}(A_{1})\rangle\Big).
\end{aligned}
\end{equation}

Analogously, since $l_{3}=H_{2}(\tilde{A}_{1,\varepsilon})-H_{3}(\tilde{A}_{1,\varepsilon})$, $l_{5}=H_{3}(\tilde{A}_{2,\varepsilon})-H_{2}(\tilde{A}_{2,\varepsilon})$ and $l_{7}=H_{2}(A_{2,\varepsilon})-H_{1}(A_{2,\varepsilon})$, using $\left.\frac{d}{d\varepsilon}\tilde{A}_{1,\varepsilon}\right|_{0}=\tilde\lambda_1\nabla\tilde\varphi_1^\perp(\tilde A_1)$, $\left.\frac{d}{d\varepsilon}\tilde{A}_{2,\varepsilon}\right|_{0}=\tilde\lambda_2\nabla\tilde\varphi_2^\perp(\tilde A_2)$ and $\left.\frac{d}{d\varepsilon}A_{2,\varepsilon}\right|_{0}=\lambda_2\nabla\varphi_2^\perp(A_2)$, we obtain
\begin{equation}
\begin{aligned}\label{a.eq14}
\left.\frac{d}{d\varepsilon}l_{3}\right|_{\varepsilon=0}
=\tilde\lambda_{1}\Big(\langle \nabla H_{2}(\tilde A_{1}),\nabla \tilde\varphi_{1}^{\perp}(\tilde A_{1})\rangle-\langle \nabla H_{3}(\tilde A_{1}),\nabla \tilde\varphi_{1}^{\perp}(\tilde A_{1})\rangle\Big),
\end{aligned}
\end{equation}
\begin{equation}
\begin{aligned}\label{a.eq15}
\left.\frac{d}{d\varepsilon}l_{5}\right|_{\varepsilon=0}
=\tilde\lambda_{2}\Big(\langle \nabla H_{3}(\tilde A_{2}),\nabla \tilde\varphi_{2}^{\perp}(\tilde A_{2})\rangle-\langle \nabla H_{2}(\tilde A_{2}),\nabla \tilde\varphi_{2}^{\perp}(\tilde A_{2})\rangle\Big),
\end{aligned}
\end{equation}
\begin{equation}
\begin{aligned}\label{a.eq16}
\left.\frac{d}{d\varepsilon}l_{7}\right|_{\varepsilon=0}
=\lambda_{2}\Big(\langle \nabla H_{2}(A_{2}),\nabla \varphi_{2}^{\perp}(A_{2})\rangle-\langle \nabla H_{1}(A_{2}),\nabla \varphi_{2}^{\perp}(A_{2})\rangle\Big).
\end{aligned}
\end{equation}

We now assemble $M(h)=\left.\dfrac{d}{d\varepsilon}(l_1+l_2+\cdots+l_8)\right|_{\varepsilon=0}$. Adding the pairs  $(l_1,l_2)$, $(l_3,l_4)$, $(l_5,l_6)$ and $(l_7,l_8)$, using (\ref{a.eq13})--(\ref{a.eq16}) together with (\ref{a.eq5}) and (\ref{a.eq7})--(\ref{a.eq9}), the terms within each pair cancel partially and we get
\begin{align*}
\left.\frac{d}{d\varepsilon}(l_1+l_2)\right|_{0} &= \lambda_1\langle\nabla H_1(A_1),\nabla\varphi_1^\perp(A_1)\rangle-\tilde\lambda_1\langle\nabla H_2(\tilde A_1),\nabla\tilde\varphi_1^\perp(\tilde A_1)\rangle,\\
\left.\frac{d}{d\varepsilon}(l_3+l_4)\right|_{0} &= \tilde\lambda_1\langle\nabla H_2(\tilde A_1),\nabla\tilde\varphi_1^\perp(\tilde A_1)\rangle-\tilde\lambda_2\langle\nabla H_3(\tilde A_2),\nabla\tilde\varphi_2^\perp(\tilde A_2)\rangle,\\
\left.\frac{d}{d\varepsilon}(l_5+l_6)\right|_{0} &= \tilde\lambda_2\langle\nabla H_3(\tilde A_2),\nabla\tilde\varphi_2^\perp(\tilde A_2)\rangle-\lambda_2\langle\nabla H_2(A_2),\nabla\varphi_2^\perp(A_2)\rangle,\\
\left.\frac{d}{d\varepsilon}(l_7+l_8)\right|_{0} &= \lambda_2\langle\nabla H_2(A_2),\nabla\varphi_2^\perp(A_2)\rangle.
\end{align*}

Adding these four lines, every intermediate term cancels and we are left with
\begin{equation}
M(h)=\lambda_1\langle \nabla H_1(A_1),\nabla\varphi_1^\perp(A_1)\rangle. \label{a.eq17}
\end{equation}

This also follows at once from the definition of the displacement function. Indeed, since $H_1(A_1(h))$ does not depend on $\varepsilon$, we have $M(h)=\left.\frac{d}{d\varepsilon}H_1(A_{1,\varepsilon}(h))\right|_0=\langle\nabla H_1(A_1),\lambda_1\nabla\varphi_1^\perp(A_1)\rangle$, which is a useful check on (\ref{a.eq17}).
Finally, substituting (\ref{a.eq12}) into (\ref{a.eq17}), then (\ref{a.eq11}) into the result, then (\ref{a.eq10})  and then (\ref{a.eq6}), and collecting the coefficients of $\int_{L_h^4}g_2dx-f_2dy$, $\int_{L_h^3}g_3dx-f_3dy$, $\int_{L_h^2}g_2dx-f_2dy$ and $\int_{L_h^1}g_1dx-f_1dy$, we obtain
\begin{equation}
\begin{aligned}
M(h) &= \Psi_{1}(h) \int_{L_{h}^{4}} (g_{2}\,dx - f_{2}\,dy)
+ \Psi_{2}(h) \int_{L_{h}^{3}} (g_{3}\,dx - f_{3}\,dy) \\
&\quad + \Psi_{3}(h) \int_{L_{h}^{2}} (g_{2}\,dx - f_{2}\,dy)
+ \Psi_{4}(h) \int_{L_{h}^{1}} (g_{1}\,dx - f_{1}\,dy),
\end{aligned}
\end{equation}
with $\Psi_{1}, \Psi_{2}, \Psi_{3}, \Psi_{4}$ exactly as in the statement of Theorem \ref{teomelnikov}. To conclude the existence of a unique limit cycle bifurcating from $L_{h_{0}}$, write $\Delta(h,\varepsilon )=H_{1}(A_{1,\varepsilon}(h))-H_{1}(A_{1}(h))$ for the displacement function, so that  $\Delta(h,\varepsilon )=\varepsilon M(h)+O(\varepsilon^2)$ by (\ref{eq.a1}). For $\varepsilon \neq 0$ define the normalized displacement $F(h,\varepsilon )=\Delta(h,\varepsilon )/\varepsilon= M(h)+O(\varepsilon)$, and, by continuity, define $F(h,0)=M(h)$. If $h_{0}$ is a simple zero of $M$, i.e.,  $M(h_{0})=0$ and $M'(h_{0})\neq0$, then $F(h_{0},0)=0$ and $\partial F/\partial h(h_{0},0)=M'(h_{0})\neq0$, so the Implicit Function Theorem applied to $F$ in the variable $h$, at the point $(h_{0},0)$, yields a unique $C^{1}$ function $h=h(\varepsilon )$, defined for $|\varepsilon|$ small, with $h(0)=h_{0}$ and $F(h(\varepsilon),\varepsilon )\equiv 0$, equivalently $\Delta(h(\varepsilon ),\varepsilon)=0$, so that for each sufficiently small $\varepsilon\neq 0$ the closed curve $L_{h({\varepsilon})}$ gives rise to a limit cycle of the perturbed system, bifurcating from $L_{h_{0}}$.

\section{Wronskian computations }\label{app.w}
This appendix collects the explicit Wronskians $W_0,\ldots,W_5$ (each expressed in terms of an explicit polynomial $p_i(r)$ or $q_i(r)$, though the Wronskians themselves are generally not polynomials --- most involve a square root) together with the corresponding zero-location analysis, for each ECT-system verification invoked in Sections \ref{sec3}--\ref{sec5} that relies on a full symbolic computation of $W_0,\ldots,W_5$ as functions of $r$. 

\subsection*{Case 1 (Theorem \ref{teoa}, statement (a))}
\begin{eqnarray*}
    W_{0}(r)&=&r\sin (\theta_{r})=r\sqrt{-\frac{p_{1}(r)}{16r^{2}}}\\
    W_{1}(r)&=&-\frac{1}{8}rp_{1}(r)\\
    W_{2}(r)&=&-\frac{1}{8}r^{4}p_{1}(r)\sqrt{-\frac{p_{1}(r)}{r^{2}}}\\
    W_{3}(r)&=&\frac{3}{8}r^{6}(p_{1}(r))^{2}\\
     W_{4}(r)&=&\dfrac{720r^{10}p_{2}(r)}{p_{1}(r)^{2}
     }\\
     W_{5}(r)&=&\dfrac{69120r^{6}p_{3}(r)}{(p_{1}(r))^{2}\sqrt{\frac{-p_{1}(r)}{r^{2}}}}
\end{eqnarray*}
where
\begin{eqnarray*}
    p_{1}(r)&=&9-10r^{2}+r^{4}\\
    p_{2}(r)&=&441-284r^{2}+78r^{4}-12r^{6}+r^{8}\\
    p_{3}(r)&=&-1323+891r^{2}-270r^{4}+54r^{6}-9r^{8}+r^{10}.
\end{eqnarray*}
On $(1,3)$: the zeros of $p_{1}(r)$ are $r=\pm1$ and $r=\pm3$, so $p_1$ (hence $W_1,W_2,W_3$) is nonzero on the open interval $(1,3)$. Writing $s=r^{2}$,  $p_{2}(r)=Q_{2}(s)=s^{4}-12s^{3}+78s^{2}-284s+441=(s^{2}-6s+14)^{2}+14(s-\frac{29}{7})^{2}+\frac{33}{7}\geq \frac{33}{7}>0$, so $p_{2} $ has no real zeros and $W_{4}$ is nonzero throughout $(1,3)$. For $p_3(r)=Q_3(s)=s^{5}-9s^{4}+54s^{3}-270s^{2}+891s-1323$, $Q_3'(s)=\frac{1}{5}\left[\left(5s^{2}-18s\right)^{2}+6\left(9s-25\right)^{2}\right]+141\geq141>0$, so $Q_{3}$ is strictly increasing on $\mathbb{R}$. Since $Q_3\!\left(\frac{18}{5}\right)=-\frac{6777}{3125}<0<\frac{1935667}{100000}=Q_3\!\left(\frac{37}{10}\right)$, the IVT gives its unique real zero at $s\approx3.61045$, so $p_{3}$ has exactly two real zeros $r\approx \pm 1.90012$, only the positive one lying in $(1,3)$. Hence $W_{0},W_{1},W_{2},W_{3},W_{4}$ are all nonzero on $(1,3)$, while $W_5$ has the single simple zero $r\approx1.90012$ described above.

\subsection*{Case 2 (Theorem \ref{teob}, statement (a))}
\begin{eqnarray*}
    W_{0}(r)&=&r\sin (\theta_{r}^{1})=\frac{r}{4}\sqrt{-\frac{q_{1}(r)}{r^{2}}}\\
    W_{1}(r)&=&-\frac{r}{8}q_{1}(r)\\
    W_{2}(r)&=&-\frac{1}{8}r^{4}q_{1}(r)\sqrt{\frac{-q_{1}(r)}{r^{2}}}\\
    W_{3}(r)&=&\frac{3}{8}r^{6}(q_{1}(r))^{2}\\
     W_{4}(r)&=&\frac{288r^{10}q_{2}(r)}{q_{1}(r)^{2}}\\
     W_{5}(r)&=&\dfrac{27648r^{6}q_{3}(r)}{(q_{1}(r))^{2}\sqrt{\frac{-q_{1}(r)}{r^{2}}}}
\end{eqnarray*}
where
\begin{eqnarray*}
    q_{1}(r)&=&r^{4}-12r^{2}+4,\\
    q_{2}(r)&=&1008-768r^{2}+216r^{4}-32r^{6}+3r^{8},\\
    q_{3}(r)&=&-2016+1040r^{2}-240r^{4}+40r^{6}-10r^{8}+3r^{10}.
\end{eqnarray*}
The zeros of $q_1(r)$ are $r=\pm\sqrt{2(3-2\sqrt2)}\approx\pm0.58579$ and $r=\pm\sqrt{2(3+2\sqrt2)}\approx\pm3.41421$, so $q_1$ (hence $W_1,W_2,W_3$) is nonzero on each of the open intervals $\left(\sqrt{2(3-2\sqrt2)},1\right)$ and $\left(3,\sqrt{2(3+2\sqrt2)}\right)$.
Writing $s=r^{2}$,  $q_2(r)=Q_2(s)=3s^{4}-32s^{3}+216s^{2}-768s+1008=\frac{1}{3}\left[\left(3s^{2}-16s+12\right)^{2}+320\left(s-3\right)^{2}\right]>0$ for all $s\in \mathbb{R}$, so $q_{2}$ has no real zero and $W_{4}$ is nonzero on both intervals. For $q_3(r)=Q_3(s)=3s^{5}-10s^{4}+40s^{3}-240s^{2}+1040s-2016$, $Q_3'(s)=\frac{5}{3}\left(3s^{2}-4s\right)^{2}+\frac{40}{21}\left(7s-18\right)^{2}+\frac{2960}{7}\geq\frac{2960}{7}>0$, so $Q_{3}$ is strictly increasing on $\mathbb{R}$. Since  $Q_3(3)=-57<0<\frac{2859353}{100000}=Q_3\!\left(\frac{31}{10}\right)$, the IVT gives its unique real zero at $s\approx 3.0678$, so $q_3(r)$ has zeros at $r\approx\pm1.75148$, neither of which lies within these two intervals, so $W_5$ is likewise nonzero there. Hence all six Wronskians $W_{0},\ldots,W_{5}$ are nonzero on each of these two intervals.

\subsection*{Case 3 (Theorem \ref{teoc}, statement (a))}
\begin{eqnarray*}
    W_{0}(r)&=&r\sin (\theta_{r}^{4})=r\sqrt{\frac{p_{1}(r)}{20r^{2}}}\\
    W_{1}(r)&=&\frac{1}{10}rp_{1}(r)\\
    W_{2}(r)&=&\frac{2}{25}r^{4}p_{1}(r)\sqrt{\frac{p_{1}(r)}{r^{2}}}\\
    W_{3}(r)&=&\frac{48}{125\sqrt{5}}r^{6}(-p_{1}(r))^{2}\\
     W_{4}(r)&=&\dfrac{36864r^{10}p_{2}(r)}{125\sqrt{125}(-p_{1}(r))^{2}
     }\\
     W_{5}(r)&=&\dfrac{3538944}{125}\dfrac{r^{6}p_{3}(r)}{(-p_{1}(r))^{2}\sqrt{\frac{5p_{1}(r)}{r^{2}}}}
\end{eqnarray*}
where
\begin{eqnarray*}
    p_{1}(r)&=&-16+12r^{2}-r^{4}\\
    p_{2}(r)&=&15840-8385r^{2}+1890r^{4}-235r^{6}+15r^{8}\\
    p_{3}(r)&=&-12672+7360r^{2}-1920r^{4}+320r^{6}-40r^{8}+3r^{10}.
\end{eqnarray*}
Observe that the zeros of $p_{1}(r)$ are $r=\pm\sqrt{6-2\sqrt5}$ and $r=\pm\sqrt{6+2\sqrt5}$, none of which belong to the open interval $\left(3,\sqrt{2(3+\sqrt{5})}\right)$. Writing $s=r^{2}$, $p_2(r)=Q_2(s)=15s^{4}-235s^{3}+1890s^{2}-8385s+15840$, its Sturm sequence has sign pattern $(+,-,-,+,+)$ at $s=0$ and $(+,+,-,-,+)$ as $s\to +\infty $, two sign changes each, so $Q_{2}$ has $2-2$ roots in $(0,\infty)$, i.e., 
 $p_{2}(r)$ does not vanish on this interval. Consequently, the functions $W_{0}(r),W_{1}(r),W_{2}(r),W_{3}(r)$, and $W_{4}(r)$ are nonzero on this interval. Similarly, $p_3(r)=Q_3(s)=3s^{5}-40s^{4}+320s^{3}-1920s^{2}+7360s-12672$, with $Q_3'(s)=\frac{5}{3}\left(3s^{2}-16s\right)^{2}+\frac{64}{3}\left(5s-18\right)^{2}+448\geq448>0$, so $Q_{3}$ is strictly increasing on $\mathbb{R}$. Since $Q_3\!\left(\frac{23}{5}\right)=-\frac{83171}{3125}<0<\frac{9218621}{100000}=Q_3\!\left(\frac{47}{10}\right)$, the IVT gives its unique real zero at $s\approx 4.6229$, so $p_{3}$
 has two simple zeros $r\approx \pm 2.1501$, neither of which lies within this interval. Hence all six Wronskians $W_{0},\ldots,W_{5}$ are nonzero on $\left(3,\sqrt{2(3+\sqrt{5})}\right)$.

\section*{Acknowledgments}
The first author is supported by CAPES grant 88887.950547/2024-00. The second author is partially supported by FAPESP grant 2023/02959-5 and 2024/15612-6, CNPq grant
302154/2022-1, and ANR-23-CE40-0028.


\begin{thebibliography}{99}

\bibitem{refb10}
S.R. Alva, A.C. Mereu, Crossing limit cycles from discontinuous piecewise linear differential centers separated by two circles, \textit{Int. J. Bifurcation Chaos} \textbf{35} (2025) 2550090.

\bibitem{refb11}
M.E. Anacleto, J. Llibre, C. Valls, C. Vidal, Limit cycles of discontinuous piecewise differential systems formed by linear centers in 
$\mathbb{R}^{2}$
 and separated by two circles, \textit{Nonlinear Anal. Real World Appl.} \textbf{60} (2021) 103291.

\bibitem{refb18}
M.E. Anacleto, C. Vidal, Limit cycles bifurcating of discontinuous and continuous piecewise differential systems of isochronous cubic centers with three zones, \textit{Qual. Theory Dyn. Syst.} \textbf{23} (2024) 173. https://doi.org/10.1007/s12346-024-01030-y.

\bibitem{refb12}
M. di Bernardo, C. Budd, A.R. Champneys, P. Kowalczyk, \textit{Piecewise-smooth Dynamical Systems: Theory and Applications}, Appl. Math. Sci., vol. 163, Springer, London, 2008.

\bibitem{refb19} C. A. Buzzi, C. Pessoa and J. Torregrosa, \emph{Piecewise linear perturbations of a linear center}, Discrete Contin. Dyn. Syst. \textbf{33} (2013), no. 9, 3915--3936.

\bibitem{ref4}
B. Coll, A. Gasull, R. Prohens, Bifurcation of limit cycles from two families of centers, \textit{Dyn. Contin. Discrete Impuls. Syst. Ser. A Math. Anal.} \textbf{12} (2) (2005) 275--287.

\bibitem{refb9}
C. V. Das Neves Silva, P. R. da Silva, On the Piecewise Holomorphic Systems with Three Zones, \textit{J. Nonlinear Sci.} \textbf{36} (2026) 1--42.

\bibitem{refb13}
A.F. Filippov, \textit{Differential Equations with Discontinuous Right-Hand Sides}, Mathematics and its Applications (Soviet Series), vol. 18, Kluwer Academic Publishers, Dordrecht, 1988.

\bibitem{refb2}
A. Garijo, A. Gasull, X. Jarque, Local and global phase portrait of equation 
$\dot{z}=f(z)$, \textit{Discrete Contin. Dyn. Syst.} \textbf{17} (2) (2007) 309--329.

\bibitem{refb6} Gasull, A., Rondón, G., Da Silva, P. R., \emph{On the number of limit cycles for piecewise polynomial holomorphic systems}, SIAM Journal on Applied Dynamical Systems, v. 23, p. 2593-2622, 2024.

\bibitem{refb1}
A. Gasull, G. Rondón, P.R. da Silva, Simultaneous bifurcation of limit cycles for piecewise holomorphic systems, \textit{J. Dyn. Differ. Equ.}, to appear (2025). https://doi.org/10.1007/s10884-025-10435-5.

\bibitem{refb5} Gouveia, L. F. S., Rondón, G., Da Silva, P. R., \emph{Piecewise holomorphic systems}, Journal of Differential Equations, v. 332, p. 440-472, 2022.

\bibitem{refb14}
M. Han, W. Chen, Melnikov function method and limit cycle bifurcations, \textit{J. Differ. Equ.} \textbf{157} (2) (1999) 287--321.

\bibitem{refb15}
Y. Ilyashenko, Centennial history of Hilbert's 16th problem, \textit{Bull. Amer. Math. Soc.} \textbf{39} (3) (2002) 301--354.

\bibitem{refb20} J. Llibre and M. A. Teixeira, \emph{Limit cycles in Filippov systems having a circle as switching manifold}, Chaos \textbf{32} (2022), no. 5, 053106.

\bibitem{refb21} J. Llibre and C. Valls, \emph{Crossing limit cycles for discontinuous piecewise differential systems formed by linear Hamiltonian saddles or linear centers separated by a conic}, Chaos, Solitons \& Fractals \textbf{159} (2022), 112076.

\bibitem{ref3}
D.D. Novaes, J. Torregrosa, On extended Chebyshev systems with positive accuracy, \textit{J. Math. Anal. Appl.} \textbf{448} (1) (2017) 171--186.

\bibitem{refb7} C. Pessoa and R. Ribeiro, \emph{Bifurcation of limit cycles from a periodic annulus formed by a center and two saddles in piecewise linear differential system with three zones}, Nonlinear Anal. Real World Appl. \textbf{80} (2024), 104153.

\bibitem{refb8}
O. Ramirez, A.M. Alves, Bifurcation of limit cycles by perturbing piecewise non-Hamiltonian systems with nonlinear switching manifold, \textit{Nonlinear Anal. Real World Appl.} \textbf{57} (2021).

\bibitem{refb17} R.M. Ribeiro, \emph{Persistence of periodic orbits from planar piecewise linear Hamiltonian differential systems with two or three zones}, Ph.D. Thesis 2023.

\bibitem{refb22} G. Rondón et al., \emph{Limit Cycles in Piecewise Smooth Systems with Circular Switching Manifold}, arXiv:2604.26061 (2026).

\bibitem{refb16}
Y. Xiong, C. Wang, Limit cycle bifurcations of planar piecewise differential systems with three zones, \textit{Nonlinear Anal. Real World Appl.} \textbf{61} (2021) 103333.

\end{thebibliography}
\end{document}